\documentclass{siamart0516}

\usepackage{braket,amsfonts}
\usepackage{array}
\usepackage{pgfplots}
\usepackage{algorithm}
\usepackage{graphicx,epstopdf}
\usepackage{amssymb}
\usepackage{pgf}
\usepackage{tikz-cd}
\usepackage{wrapfig}
\usepackage{float}
\usepackage{caption}
\usepackage{chngcntr}
\usepackage{apptools}
\usepackage[noend]{algpseudocode}
\usepackage{wrapfig}

\newsiamthm{claim}{Claim}
\newsiamremark{rem}{Remark}
\newsiamremark{expl}{Example}
\newsiamremark{hypothesis}{Hypothesis}
\crefname{hypothesis}{Hypothesis}{Hypotheses}
\newsiamremark{assumption}{Assumption}

\Crefname{ALC@unique}{Line}{Lines}

\usepackage{amsopn}

\newcommand{\real}{{\mathbb{R}}}
\newcommand{\realpositive}{{\mathbb{R}}_{>0}}
\newcommand{\realnonnegative}{{\mathbb{R}}_{\ge 0}}

\newcommand{\Sc}{\mathcal{S}}

\newcommand{\longthmtitle}[1]{\mbox{}{\bf \textit{(#1).}}}
\newcommand{\subscr}[2]{#1_{\text{#2}}}
\newcommand{\supscr}[2]{#1^{\textup{#2}}}

\newcommand\oprocendsymbol{\hbox{$\bullet$}}
\newcommand\oprocend{\relax\ifmmode\else\unskip\hfill\fi\oprocendsymbol}

\numberwithin{theorem}{section}

\title{Distributed Control for Spatial Self-Organization of Multi-Agent Swarms%
 			\thanks{This work has been partially supported by grant FA9550-18-1-0158.}}

\author{Vishaal Krishnan and Sonia Mart\'inez %
			  \thanks{The authors are with the Department of Mechanical and Aerospace Engineering, University of California at San Diego, La Jolla CA 92093 USA (email: v6krishn@ucsd.edu; soniamd@ucsd.edu).}}
			  
\begin{document}

\maketitle

\begin{keywords}
	Self-organization, Distributed control, Pseudo-localization, Harmonic maps
\end{keywords}

\begin{AMS}
  34B45, 35B40, 35B35, 58J32, 58J35, 58E20 
\end{AMS}

\begin{abstract}
  In this work, we design distributed control laws for spatial
  self-organization of multi-agent swarms in 1D and 2D spatial
  domains. The objective is to achieve a desired density distribution
  over a simply-connected spatial domain. Since individual agents in a
  swarm are not themselves of interest and we are concerned only with
  the macroscopic objective, we view the network of agents in the
  swarm as a discrete approximation of a continuous medium and design
  control laws to shape the density distribution of the continuous
  medium. The key feature of this work is that the agents in the swarm
  do not have access to position information. Each individual agent is
  capable of measuring the current local density of agents and can
  communicate with its spatial neighbors. The network of agents
  implement a Laplacian-based distributed algorithm, which we call
  pseudo-localization, to localize themselves in a new coordinate
  frame, and a distributed control law to converge to the desired
  spatial density distribution. We start by studying self-organization
  in one-dimension, which is then followed by the two-dimensional
  case.
\end{abstract}

\section{Introduction}
Self-organization in swarms refers broadly to the emergence of
patterns of long-range order in large groups of dynamic agents
which interact locally with each other.  It is a
pervasive phenomenon in nature, observed in biological~\cite{SC:03}
and other natural systems~\cite{MGW-BG:02}. 
In the context of robotic systems, problems of 
deployment and formation control of groups of robots have been extensively 
studied~\cite{FB-JC-SM:09, MM-ME:10,JC-SM-TK-FB:02j, MS-DR-JJS:09, AH-MJM-GSS:02}.
More recently, research efforts have been undertaken to massively increase the scale of these robotic 
systems~\cite{MR-CA-NH-AC-RN:14}.  
This transition does not
merely involve an increase in the size of robotic networks, but it
also introduces new theoretical challenges for their analysis and
control design. In particular, large groups of agents have some
essential characteristics that distinguish them from other
smaller-scale counterparts. In a swarm, individual agents have no
significance and only the macroscopic objectives are relevant. A swarm
largely remains unaffected by the removal of a large, but discrete,
number of agents. Moreover, it is
difficult (and needlessly complicated) to specify the global
configuration of the swarm using the states of individual agents;
instead, employing macroscopic quantities such as the swarm spatial
density distribution to specify its configuration is more appropriate.
From an analysis and control-theoretic viewpoint, the dynamic modeling
of swarms is less explored, which e.g.~can be established by means of
PDEs, for which control theoretic tools are less well developed in
comparison to ODEs. These theoretical challenges motivate the
investigation of self-organization in large-scale swarms.

In the literature, Markov-chain based methods have been widely used in
addressing some of the key theoretical problems pertaining to swarm
self-organization.  By means of it, the swarm configuration is
described through the partitioning the spatial domain in a finite
number of larger size disjoint subregions, on which a probability
distribution is defined. Then, the self-organization problem is
reduced to the design of the transition matrix governing the evolution
of this probability density function to ensure its convergence to a
desired profile. A recent approach to density control using Markov
chains is presented in~\cite{ND-UE-BA:15}, which includes additional
conflict-avoidance constraints. In this setting every agent is able to
determine the bin to which it belongs at every instant of time, which
essentially means that individual agents have self-localization
capabilities. Also, the dimensional transition matrix is synthesized
in a central way at every instant of time by solving a convex
optimization problem.  In~\cite{SB-SJC-FYH:13}, the authors make use
of inhomogeneous Markov chains to minimize the number of transitions
to achieve a swarm formation. In this approach, the algorithm
necessitates the estimation of the current swarm distribution, and
computes the transition Markov matrices for each agent, at each
instant of time. The fact that every agent needs to have an estimate
of the global state (swarm distribution) at every time may not be
desirable or feasible. The localization of each agent still remains to
be a main assumption. Under similar conditions, one can find the
manuscripts~\cite{BA-DSB:15} and~\cite{IC-AR:09}, which describe
probabilistic swarm guidance algorithms. In~\cite{SB-AH-MAH-VK:09},
the authors present an approach to task allocation for a homogeneous
swarm of robots. This is a Markov-chain based approach, where the goal
is to converge to the desired population distribution over the set of
tasks.

In the context of robotic swarms, programmable self-assembly of
two-dimensional shapes with a thousand-robot swarm is demonstrated
in~\cite{MR-AC-RN:14}. These robots are capable of measuring distances
to nearby neighbors which they use to localize themselves relative to
other localized robots. Each robot then uses its position to implement
an edge-following algorithm.

Another approach uses partial differential equations to model swarm
behaviour, and control action is applied along the boundary of the
swarm. Previous works on PDE-based methods with boundary control
include~\cite{PF-MK:11}, where the authors present an algorithm for
the deployment of agents onto families of planar curves. Here, the
swarm collective dynamics are modeled by the
reaction-advection-diffusion PDE and the particular family of curves
to which the swarm is controlled to is parametrized by the continuous
agent identity in the interval of unit length. An extension of this
work to deployment on a family of $2$D surfaces in $3$D space can be
found in~\cite{JQ-RV-MK:15}. The problem of planning and task allocation
is addressed in the framework of advection-diffusion-reaction PDEs in~\cite{KE-SB:15}.
In~\cite{GF-SF-TAW:14} and~\cite{SF-GF-PZ-TAW:16},
the authors present an optimal control problem formulation for
swarm systems, where microscopic control laws are derived from the
optimal macroscopic description using a potential function approach.

The problem of position-free extremum-seeking of an external scalar
signal using a swarm of autonomous vehicles, inspired by bacterial
chemotaxis, has been studied in~\cite{AM-JH-KA:08}.

In this work, we adopt a viewpoint outlined in~\cite{JB-JB-JM:10},
wherein we make an amorphous medium abstraction of the swarm, which is
essentially a manifold with an agent located at each point.  We then
model the system using PDEs and design distributed control laws for
them. An important component of this paper is the Laplacian-based
distributed algorithm which we call pseudo-localization algorithm,
which the agents implement to localize themselves in a new coordinate
frame. The convergence properties of the graph Laplacian to the
manifold Laplacian have been studied in~\cite{MB-PN:08}, which find
useful applications in this paper.

The main contribution of this paper is the development of distributed
control laws for the index- and position-free density control of
swarms to achieve general 1D and a large class of 2D density
profiles. In very large swarms with thousands of agents, particularly
those deployed indoors or at smaller scales, presupposing the
availability of position information or pre-assignment of indices to
individual agents would be a strong assumption. In this paper, in
addition to not making the above assumptions, the agents are only
capable of measuring the local density, and in the $2$D case, the
density gradient and the normal direction to the boundary.

Under these assumptions, we present distributed pseudo-localization
algorithms for one and two dimensions that agents implement to compute
their position identifiers. Since every agent occupies a unique
spatial position, we are able to rigorously characterize the resulting
position assignment as a one-to-one correspondence between the set of
spatial coordinates and the set of position identifiers, which
corresponds to a diffeomorphism of the continuum domain. Based on this
assignment, we then design control strategies for self-organization in
one and two dimensions under the assumption that the motion control of
agents is noiseless. The extension to the $2$D case leads to new
difficulties related to the control of the swarm boundaries. To
address these, we implement a variant of the $1$D pseudo-localization
algorithm  at the boundary during an initialization
phase. A preliminary version of this work appeared
in~\cite{VK-SM:16-mtns} where we presented an outline of the
algorithms and stated some of the results. We develop them here
rigorously, providing detailed proofs for our claims.

The paper is organized as follows. In Section~\ref{sec:prelim}, we
introduce the basic notation and preliminary concepts used in the
manuscript. We present the analysis of self-organization in one
dimension in Section~\ref{sec:1D_Self_Organization}, where we
introduce the pseudo-localization algorithm in
Section~\ref{sec:pseudoloc_1D} and the distributed control law in
Section~\ref{subsec:distributed control}. After this, we generalize
and extend the analysis for self-organization in two dimensions in
Section~\ref{sec:self-organization in
  2D}. Section~\ref{sec:simulation} contains numerical simulations of
the results in the paper, and in Section~\ref{sec:conclusions}, we
present our conclusions.
\section{Preliminaries}
\label{sec:prelim}
Let $\real$ denote the set of all real numbers, $\realnonnegative$ the
set of non-negative real numbers, and $\real^n$ the~$n$-dimensional
Euclidean space. We use boldface letters to denote vectors in
$\real^n$.  The norm $| \mathbf{x} |$ of a vector $\mathbf{x} \in
\real^n$ is the standard Euclidean $2$-norm, unless otherwise
specified. Let $\nabla = \left( \frac{\partial}{\partial x_1}, \ldots
  \frac{\partial}{\partial x_n} \right)$ denote the gradient operator
in $\real^n$ when acting on real-valued functions and the Jacobian in
the context of vector-valued functions.  As a shorthand, we let
$\frac{\partial}{\partial z}(\cdot) =
\partial_z (\cdot)$ for a variable~$z$.  Let $\Delta = \sum_{i=1}^n
\frac{\partial^2}{\partial x_i^2}$ be the Laplace operator
in~$\real^n$.  We denote by either $\dot{S}$ or $\frac{dS}{dt}$ the
total time derivative of $S(t)$. Given functions $f,g : \real
\rightarrow \real$, we write $f = \mathcal{O}(g)$ if there exist
positive constants~$C$ and~$c$ such that $|f(h)| \leq C |g(h)|$, for
all $|h| \leq c$. Let $\Sc$ denote the set of agents in the swarm, and
$N$ its cardinality. For the $1$D case, let $l \in \Sc$ denote the
leftmost agent, and $r \in \Sc$ the rightmost one. Let $\mathcal{N}_i$
denote the spatial neighborhood of agent~$i$, which comprises those
agents located inside a small ball centered at $i$. A set-valued
mapping, denoted by $f: \real \rightrightarrows \real^2$, maps the set
of real numbers onto subsets of~$\real^2$.  For a bounded open set
$\Omega \subset \real^n$, $\partial \Omega$~denotes its boundary,
$\bar{\Omega} = \Omega \cup \partial \Omega$ its closure and
$\mathring{\Omega} = \Omega \setminus \partial \Omega$ its interior
with respect to the standard Euclidean topology. The set of smooth
real-valued functions on~$\Omega$ is denoted by $C^{\infty}(\Omega)$.
We let~$\mu$ (or $dx$ in 1D) denote the standard Lebesgue measure;
with a slight abuse of notation, we sometimes omit $d\mu$ (resp.~$dx$
in 1D) from long integrals. The Dirac measure~$\delta$ on~$\Omega$
defined for any~$x \in \Omega$ and any measurable set $D \subseteq
\Omega$ is given by $\delta_x(D) = 1$ for $x \in D$, and $\delta_x(D)
= 0$ for $x \notin D$.

For two non-empty subsets~$M_1$ and~$M_2$ of a metric space $(M,d)$,
the Hausdorff distance $d_H(M_1,M_2)$ between them is defined as:
\begin{align}
  d_H(M_1,M_2) = \max \lbrace \sup_{x \in M_1} \inf_{y \in M_2} d(x,y)
  , \sup_{y \in M_2} \inf_{x \in M_1} d(x,y) \rbrace.
	\label{eq:defn_Hausdorff}
\end{align}
On a measurable space~$U$, let~$L^p(U) =
\lbrace f: U \rightarrow \real \,|\, \|f\|_{L^p(U)} = \left( \int_{U} |f|^p d\mu
\right)^{1/p} < \infty\rbrace$ constitute the~$L^p$ space, where~$\|
\cdot \|_{L^p(U)}$ is the~$L^p$ norm. Of particular interest is
the~$L^2$ space, or the space of square-integrable functions. 
In this paper, we denote by $\| f \|_{L^2(U)}$ the $L^2$ norm of~$f$
with respect to the Lebesgue measure, and by $\| f \|_{L^2(U, \rho)}$ the 
weighted $L^2$ norm (with the strictly positive weight~$\rho$ on~$U$).
 The Sobolev space~$W^{1,p} (U)$ over a measurable space~$U$ is defined
as~$W^{1,p} (U) = \lbrace f: U \rightarrow \real \,|\, \| f \|_{W^{1,p}} =
\left( \int_{U} |f|^p + \int_{U} |\nabla f|^p \right)^{1/p} < \infty
\rbrace$. Of particular interest is the space~$W^{1,2}$, also called
the $H^1$ space. For two functions $f(t, \cdot)$ and $g(\cdot)$, we 
denote by $f \rightarrow_{L^2} g$ the convergence in $L^2$ norm (over the domain~$U$ of the functions) of $f(t,\cdot)$ to
$g(\cdot)$ as $t \rightarrow \infty$, that is, $\lim_{t \rightarrow \infty} \| f(t, \cdot) - g(\cdot) \|_{L^2} = 0$.
Convergence in $H^1$ norm is denoted similarly by $f \rightarrow_{H^1} g$.

We now state some well-known results that we will be used in the
subsequent sections of this paper.
\begin{lemma}\longthmtitle{Divergence Theorem~\cite{AJC-JEM:90}} \label{le:dt}
  For a smooth vector field $\mathbf{F}$ over a bounded open set
  $\Omega \subseteq \real^n$ with boundary $\partial \Omega$, the
  volume integral of the divergence $\nabla \cdot \mathbf{F}$ of
  $\mathbf{F}$ over $\Omega$ is equal to the surface integral of
  $\mathbf{F}$ over $\partial \Omega$:
\begin{align}
  \int_{\Omega} (\nabla \cdot \mathbf{F})~d\mu = \int_{\partial \Omega}
  \mathbf{F} \cdot \mathbf{n}~dS,
	\label{eqn:divergence_thm} 
\end{align}
where $\mathbf{n}$ is the outward normal to the boundary and $dS$ the
measure on the boundary. For a scalar field~$U$ and a vector
field~$\mathbf{F}$ defined over~$\Omega \subseteq \real^n$:
\begin{align*}
  \int_{\Omega} (\mathbf{F} \cdot \nabla U)~d\mu= \int_{\partial
    \Omega} U (\mathbf{F} \cdot \mathbf{n})~dS - \int_{\Omega} U
  (\nabla \cdot \mathbf{F})~d\mu.
\end{align*}
\end{lemma}
%
\begin{lemma}\longthmtitle{Leibniz Integral Rule \cite{AJC-JEM:90}}
\label{lemma:Leibniz_rule}
Let $f \in \mathcal{C}^{\infty}(\real \times \real^n)$ and $\Omega :
\real \rightrightarrows \real^n$ be a smooth
one-parameter family of bounded open sets in~$\real^n$ generated by
the flow corresponding to the smooth vector field~$\mathbf{v}$
on~$\real^n$. Then:
\begin{align*}
  \frac{d}{dt} \left( \int_{\Omega(t)} f(t, \mathbf{r})~d\mu \right) =
  \int_{\Omega(t)} \partial_t( f(t, \mathbf{r}))~d\mu + \int_{\partial
    \Omega(t)} f(t, \mathbf{r}) \mathbf{v} \cdot \mathbf{n}~dS.
\end{align*} 
\end{lemma}
\begin{corollary}\longthmtitle{Derivative of Energy Functional}
\label{lemma:time_der}
Let $U$ be an energy functional defined as follows:
\begin{align*}
	U = \frac{1}{2} \int_{\Omega} |f|^2~d\mu,
\end{align*}
for some function $f: \Omega \rightarrow \real$.  Then,
\begin{align*}
  \dot{U} = \int_{\Omega} f \cdot \left(\frac{df}{dt}
  \right)~d\mu + \frac{1}{2} \int_{\Omega} |f|^2 \nabla \cdot
  \mathbf{v}~d\mu.
\end{align*}
where $\frac{d}{dt} = \partial_t + \mathbf{v} \cdot \nabla$ is the
total derivative.
\begin{proof}
  We have included the proof for this corollary for the sake of
  completeness. Using the Leibniz integral rule and the Divergence
  theorem, we have (it is understood that the integrations are with
  respect to the measure~$\mu$):
\begin{align*}
  \frac{\partial U}{\partial t} &= \int_{\Omega} f \cdot f_t +
  \frac{1}{2} \int_{\partial \Omega} |f|^2
  \mathbf{v} \cdot \mathbf{n} \\
  &= \int_{\Omega} f \cdot f_t + \frac{1}{2} \int_{\Omega} \nabla \cdot (|f|^2 \mathbf{v}) \\
  &= \int_{\Omega} f \cdot f_t + \int_{\Omega} f \cdot (\mathbf{v}
  \cdot \nabla)
  f  + \frac{1}{2} \int_{\Omega} |f|^2 \nabla \cdot \mathbf{v} \\
  &= \int_{\Omega} f \cdot ( f_t + (\mathbf{v} \cdot \nabla) f ) +
  \frac{1}{2} \int_{\Omega} |f|^2 \nabla \cdot \mathbf{v} \\
  &= \int_{\Omega} f \cdot \left(\frac{df}{dt} \right) + \frac{1}{2}
  \int_{\Omega} |f|^2 \nabla \cdot \mathbf{v}.
\end{align*}
\end{proof}
\end{corollary}
\begin{lemma}\longthmtitle{Poincar\'{e}-Wirtinger Inequality~\cite{GL:09}}
\label{lemma:poincare_wirtinger}
  For $p \in [1, \infty]$ and $\Omega$, a bounded connected open subset
  of $\real^n$ with a Lipschitz boundary, there exists a constant $C$
  depending only on $\Omega$ and $p$ such that for every function $u$
  in the Sobolev space $W^{1,p}(\Omega)$:
\begin{align*}
  \|u-u_{\Omega}\|_{L^p(\Omega)} \leq C \|\nabla u\|_{L^p(\Omega)},
\end{align*}
where $u_{\Omega} = \frac{1}{|\Omega|} \int_{\Omega} u d\mu$, and
$|\Omega|$ is the Lebesgue measure of $\Omega$.
\end{lemma}
\begin{lemma}\longthmtitle{Rellich-Kondrachov Compactness
    Theorem~\cite{LCE:98}}
\label{lemma:RKCT}
Let $U \subset \real^n$ be open, bounded and such that~$\partial U$
is~$C^1$.  Suppose~$1 \leq p < n$, then~$W^{1,p} (U)$ is compactly
embedded in $L^q(U)$ for each~$1 \leq q < \frac{pn}{n - p}$. In particular,
we have~$W^{1,p}(U)$ is compactly contained in $L^p(U)$.
\end{lemma}
\begin{lemma}\longthmtitle{LaSalle Invariance Principle~\cite{DH:81, JAW:13, JAW:78}}
\label{lemma:LaSalle_inv}
Let~$\lbrace \mathcal{P}(t)\,|\, t \in \realnonnegative \rbrace$ be a
continuous semigroup of operators on a Banach space~$U$ (closed subset of a
Banach space with norm~$\| \cdot \|$), and for any~$u \in U$, define
the positive orbit starting from~$u$ at~$t=0$ as $\Gamma_{+}(u) =
\lbrace \mathcal{P}(t)u  \,|\, t \in \realnonnegative \rbrace
\subseteq U$.
 Let~$V: U \rightarrow \real$ be a continuous Lyapunov
functional on~$\mathcal{G} \subset U$ for~$\mathcal{P}$ (such that~$\dot{V}(u) = \frac{d}{dt} V(\mathcal{P}(t)u) \leq 0$ in~$\mathcal{G}$).  Define~$E =
\lbrace u \in \bar{\mathcal{G}} \,|\, \dot{V}(u) = 0 \rbrace$, and let~$\tilde{E}$ be the
largest invariant subset of~$E$. If for~$u_0 \in \mathcal{G}$, the
orbit~$\Gamma_{+}(u_0)$ is pre-compact (lies in a compact subset
of~$U$), then~$\lim_{t \rightarrow +\infty} d_U(\mathcal{P}(t)u_0,
\tilde{E}) = 0$, where~$d_U(y,\tilde{E}) = \inf_{x \in \tilde{E}} \| y - x \|_U$ 
(where~$d_U$ is the distance in~$U$).
\end{lemma}
\subsection{Continuum model of the swarm}
Given that $N$, the number of agents in the swarm, is very large, we
will analyze the swarm dynamics through a continuum approximation. Let
$t \in \realnonnegative$, and let $M: \real \rightrightarrows \real^n$ be a
smooth one-parameter family of bounded open sets, such that the agents
are deployed over~$\bar{M}(t)$ at time~$t$. We denote by
$\dot{\mathbf{r}}_i(t) = \mathbf{v}_i$, $\forall i \in \Sc$, where
$\mathbf{r}_i(t) \in \bar{M}(t)$ is the position of the~$i$th agent in
the swarm at time~$t$.  Let $\rho : \real_{\geq 0} \times \real^n
\rightarrow \real_{\geq 0}$ be the spatial density function supported
on~$\bar{M}(t)$ for all~$t \geq 0$ (with $\rho(t, \mathbf{r}) > 0$ for
$\mathbf{r} \in \bar{M}(t)$), such that $\int_{M(t)} \rho(t,
\mathbf{r}) d\mathbf{\mu} = 1$. We assume that~$M(t)$ is simply
connected and that the boundary $\partial M(t)$ does not
self-intersect for all~$t \geq 0$.

Assuming that $\rho$ is smooth, the macroscopic dynamics can
now be described by the continuity equation \cite{AJC-JEM:90},
assuming that the total number of agents is conserved:
\begin{align}
  \frac{\partial \rho}{\partial t} + \nabla \cdot (\rho \mathbf{v}) = 0,
  \hspace{0.2in} \forall~\mathbf{r} \in \mathring{M}(t),
	\label{eqn_continuity}
\end{align}
where $\mathbf{v}: \real_{\geq 0} \times \real^n \rightarrow \real^n$
is the velocity field with $\mathbf{v}_i(t) =
\mathbf{v}(t,\mathbf{r}_i)$, such that the one-parameter family~$M$ is
generated by the flow associated with~$\mathbf{v}$.
\subsection{Harmonic maps and diffeomorphisms}
Let $(M,g)$ and $(N,h)$ be two Riemannian manifolds of dimensions $m$
and $n$, and Riemannian metrics $g$ and $h$, respectively. A map $\phi
: M \rightarrow N$ is called harmonic if it minimizes the functional:
\begin{align}
	E(\phi) = \int_{M} | \nabla \phi|^2 dv_g,
	\label{eqn:harmonic_energy}
\end{align}
where $dv_g$ is the Riemannian volume form on $M$.
The Euler-Lagrange equation for the functional $E$, which also yields
the minimum energy, is given by $\Delta \phi = 0$, the Laplace
equation \cite{FH:02}.  It is useful to note that the solutions to the
heat equation, in the limit $t \rightarrow \infty$, approach the
harmonic map. This is proved later in Lemma~\ref{lemma:stage2_lemma},
and forms the basis for the design of the distributed
pseudo-localization algorithm.  We now state a lemma on harmonic
diffeomorphisms of Riemann surfaces (i.e., $m = n = 2$ above).
\begin{lemma} \longthmtitle{Harmonic diffeomorphism \cite{JE-LL:81}}
\label{lemma:harmonic_diffeomorphism}
Let $(M,g)$ be a compact surface with boundary and $(N,h)$ a compact
surface with non-positive curvature. Suppose that $\psi:M \rightarrow
N$ is a diffeomorphism onto $\psi(M)$. Assume that $\psi(M)$ is
convex. Then there is a unique harmonic map $\phi:M \rightarrow N$
with $\phi = \psi$ on~$\partial M$, such that $\phi: {M} \rightarrow
\phi(M)$ is a diffeomorphism.
\end{lemma}
We note that the non-positive curvature constraint in the lemma is
essentially a constraint on the metric $h$ on $N$, and the curvature
is zero for the Euclidean metric.

\section{Problem description and conceptual approach}
\label{sec:conceptual}
In this section, we provide a high-level description of the proposed
problem and explain the conceptual idea behind our approach. The
technical details can be found in the following sections. 

The problem at hand is to ultimately design a distributed control law
for a swarm to converge to a desired configuration. Here, a swarm
configuration is a density function $\rho$ of the multi-agent system
and the objective is that agents reconfigure themselves into a desired
known density $\rho^*$. To do this, an agent at position $x$ is able
to measure the current local density value, $\rho(t, x)$; however, its
position~$x$ within the swarm is unknown. Thus, given $\rho^*$, an
agent at $x$ cannot directly compute $\rho^*(x)$ nor a feedback law
based on $\rho - \rho^*$.  To solve this problem, we devise a
mechanism that allows agents to determine their coordinates in a
distributed way in an equivalent coordinate system.

Note that, given a diffeomorphism $\Theta^*$ from the spatial domain
of the swarm onto the unit interval or disk (i.e.~a coordinate
transformation), we can equivalently provide the agents with a
transformed density function $p^*$, such that $p^* = \rho^* \circ
(\Theta^{*})^{-1}$.  In this way, instead of $\rho^*$ the agents are
given $p^*$, but still do not have access to $\Theta^*$. The
pseudo-localization algorithm is a mechanism that agents employ to
progressively compute an appropriate (configuration-dependent)
diffeomorphism by local interactions.

In 1D, the pseudo-localization algorithm is a continuous-time PDE
system in a new variable or pseudo-coordinate $X$ which plays the role
of an ``approximate $x$ coordinate'' that agents can use to know where
they are. The input to this system is the current density value
$\rho$, see Figure~\ref{fig:interconnection} for an illustration, and
the objective is that $X$ converges to a $\rho$-dependent
diffeomorphism. On the other hand, the variable $X$ and the function
$p^*$ are used to define the control input of another PDE system in
the density $\rho$. In this way, we have a feedback interconnection of
two systems, one in $X$ and one in $\rho$, with the goal to achieve $X
\rightarrow \Theta^*$ (the pseudo-coordinate $X$ converges to a true
coordinate given by $\Theta^*$) and $\rho \rightarrow \rho^*$.
\begin{figure}[!h]
	\begin{center}
	\includegraphics[width=0.7\textwidth]{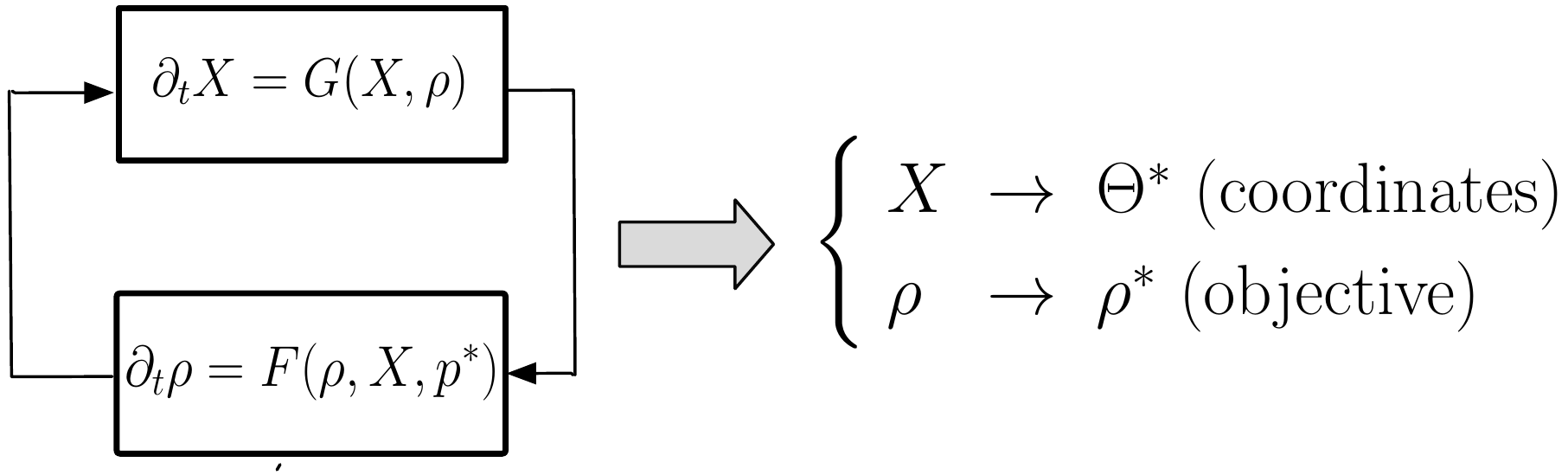}
	\captionsetup{justification=centering}
	\end{center}
	\caption{Feedback interconnection of pseudo-localization
          system in $X$ and system in $\rho$ in the 1D case. The
          function $p^* $ is an equivalent density objective provided
          to agents in terms of a diffeomorphism $\Theta^*$. The
          variables $X$ play the role of coordinates and eventually
          converge to the true coordinates given by $\Theta^*$.
          Agents use $p^*$ and $X$ to compute the control in the
          equation $\rho$. In turn, agents move and this will require
          a re-computation of coordinates or update in $X$. The control
          strategy in the 2D case (stages 2 and 3) can be interpreted
          similarly. }
	\label{fig:interconnection}
\end{figure}

As for the control design methodology, we follow a
constructive, Lyapunov-based approach to designing distributed control
laws for the swarm dynamics modeled by PDEs.  For this, we define
appropriate non-negative energy functionals that encode the objective
and choose control laws that keep the time derivative of the energy
functional non-positive. This, along with well-known results on the
precompactness of solutions as in Lemma~\ref{lemma:RKCT}, the Rellich
Kondrachov compactness theorem, allows us to apply the LaSalle
Invariance Principle in Lemma~\ref{lemma:LaSalle_inv} and other
technical arguments to establish the convergence results that we seek.

In the 1D case, we can identify a set of diffeomorphisms $\Theta$
associated with any $\rho$ that eventually converge to $\Theta^*$, and
simultaneously control boundary agents into a desired final domain
(the support of $\rho^*$). These are given by the cumulative
distribution function associated with the density function; see
Section~\ref{sec:pseudoloc_1D}.  The 2D case is more complex, and
analogous results could not be derived in their full generality.
Unlike the 1D case, estimating the cumulative distribution is not straightforward in the 2D case. 
Instead, we set out to find
diffeomorphisms as the result of a distributed algorithm. Given that
the discretization of heat flow naturally leads to distributed
algorithms, we investigate under what conditions this is the case via
harmonic map theory.  On the control side, there also are additional
difficulties, and because of this, we simplify the control strategy
into three stages.  In the first stage, the boundary agents are
re-positioned onto the boundary of the desired domain while containing
the others in the interior. Once this is achieved, the second and
third stages can be seen again as the interconnection of two systems
in pseudo-coordinates $R=(X,Y)$ (instead of $X$) and $\rho$,
analogously to Figure~\ref{fig:interconnection}. However, we apply a
two time-scale separation for analysis by which coordinates are
computed in a fast-time scale and reconfiguration is done in a
slow-time scale, which allows for a sequential analysis of the two
stages. We then study the robustness of this approach. We refer the reader to
the extended version of this paper~\cite{VK-SM:17} for further description of the 
discrete implementation.
\section{Self-organization in one dimension}
\label{sec:1D_Self_Organization}
In this section, we present our proposed pseudo-localization algorithm and
the distributed control law for the $1$D self-organization problem.

For each $t \in \realnonnegative$, let $M(t) = (0,
L(t)) \subset \real$ be the interval (with boundary~$\lbrace 0, L(t) \rbrace$) in which the agents are
distributed in 1D, and let $\rho: \real \times \real \rightarrow
\real_{\geq 0}$ be the normalized density function supported on $\bar{M}(t)$,
for all $t \geq 0$ (with $\rho(t, x) > 0$, $\forall x \in \bar{M}(t)$),
describing the swarm on that interval. Without loss of generality, we
place the origin at the leftmost agent of the swarm. We also assume
that the leftmost and the rightmost agents,~$l$ and~$r$, are aware
that they are at the boundary. Let $\rho^*: \bar{M}^* = [0,L^*] \rightarrow
\realpositive$ be the desired normalized density distribution. 

Since a direct feedback control law can not be implemented by agents
because they do not have access to their positions, we introduce an
equivalent representation of the density $\rho^*$, $p^*$, depending on
a particular diffeomorphism $\Theta^*$. First, define $\Theta^*: \bar{M}^*
\rightarrow [0,1]$ such that $\Theta^*(x) = \int_0^{x}
\rho^*(\bar{x})d\bar{x}$ and $\Theta^*(L^*) = 1$.

Now, let $p^*: [0,1] \rightarrow \realpositive$, and
$\theta^* \in \Theta^*(\bar{M}^*) = [0,1]$, be such that $p^*(\theta^*) =
\rho^*((\Theta^*)^{-1}(\theta^*)) = \rho^*(x)$.
   \begin{figure}[H]
   \centering
   	\begin{tikzcd}
             &\rho^*(x) = p^*(\theta^*) \\
             x \in [0, L^*] \arrow[ur, mapsto, shift left, start anchor =
             {[xshift = -4ex]}, end anchor = {[xshift = -1.5ex]},
             "\rho^*"] \arrow[r, mapsto, end anchor = {[xshift = 6ex]},
             "\Theta^*"] & \hspace{1cm} \Theta^*(x) = \theta^* \in [0,1]
             \arrow[u, mapsto, shift right = 4ex, "p^*"]
   	\end{tikzcd}
   \end{figure}
The function $p^*$, which represents the desired density distribution
mapped onto the unit interval~$[0,1]$, is computed offline and is
broadcasted to the agents prior to the beginning of the
self-organization process. We use~$p^*$ to derive the distributed
control law which the agents implement.  We assume that~$p^*$ is a
Lipschitz function in the sequel. 
\begin{assumption}\longthmtitle{Uniform boundedness of density function}
We assume that the density function and its derivative are uniformly bounded in its support, that is,
for~$\rho(t, \cdot)$ and~$\partial_x \rho (t, \cdot)$ there exist uniform lower bounds $d_l, D_l$
and uniform upper bounds $d_u, D_u$ (where $0< d_l \leq d_u < \infty$ and~$0< D_l \leq D_u < \infty$)
(that is, $d_l \leq \rho(t,x) \leq d_u$ for all~$t \in \realnonnegative$ and $x \in  [0,L(t)]$ and~$D_l \leq \partial_x \rho(t,x) \leq D_u$
for all~$t \in \realnonnegative$ and $x \in  (0,L(t))$).
\label{assumption:rho}
\end{assumption}
\subsection{Pseudo-localization algorithm in one dimension}
\label{sec:pseudoloc_1D}

We first consider the static case, that is, the design of the
pseudo-localization dynamics on $X$ of the upper block in
Figure~\ref{fig:interconnection}, when the agents and $\rho$ are
stationary. We define $\Theta : \bar{M} = [0,L] \rightarrow [0,1]$ as:
\begin{align}
	\Theta(x) = \int_0^x \rho(\bar{x})d\bar{x},
	\label{eq:theta}
\end{align}
such that $\Theta(L) = 1$. In other words, $\Theta$ is the cumulative
distribution function (CDF) associated with $\rho$. (Note that the
domains are static and hence the argument~$t$ has been dropped, which
will be reintroduced later.)

\begin{lemma}\longthmtitle{The CDF diffeomorphism}
\label{diffeo_lemma}
Given $\rho: \bar{M} \rightarrow \realpositive$, a~$C^1$ function, the
mapping $\Theta : \bar{M} \rightarrow [0,1]$ as defined above, is a
diffeomorphism and $\Theta(\bar{M}) = [0,1]$.
\end{lemma}
\begin{proof}
  Since $\rho(x) > 0$, $\forall x \in \bar{M}$, it follows that $\Theta$ is a
  strictly increasing function of $x$, and is therefore a one-to-one
  correspondence on $\bar{M}$. Moreover, $\Theta$ is atleast~$C^1$ and has a
  differentiable inverse, which implies it is a
  diffeomorphism. Finally, since $\Theta(L) = 1$, we have $\Theta(\bar{M}) =
  [0,1]$.
\end{proof}
Our goal here is to set up a partial differential equation with
appropriate boundary conditions that yield the diffeomorphism~$\Theta$
as its asymptotically stable steady-state solution. We begin by
setting up the pseudo-localization dynamics for a stationary swarm (for which
the spatial domain $M$ and the density distribution $\rho$ are
fixed). Let $X: \real \times \bar{M} \rightarrow \real$ be such that $(t,x)
\mapsto X(t,x) \in \real$, with:
\begin{align}
	\begin{aligned}
          &\partial_t X= \frac{1}{\rho} \partial_x \left( \frac{\partial_x X}{\rho} \right), \\
          &X(t,0) = \alpha(t), \hspace{0.2in} X(t,L) = \beta(t), \hspace{0.2in} X(0,x) = X_0(x),\\
          & \dot{\alpha} (t) =  - \alpha(t),\hspace{0.2in} \dot{ \beta }(t) = 1 - \beta(t),
	\end{aligned}
	\label{pseudoloc_cont}
\end{align}
where $\alpha: \real \rightarrow \real$ is a control input at the
boundary~$x=0$ and $\beta: \real \rightarrow \real$ is a control input
at the boundary~$x = L$.  From~\eqref{eq:theta}, we observe that
$\partial_x \left( \frac{\partial_x \Theta}{\rho} \right) =
0$. Letting $w = X - \Theta$ denote the error, we obtain:
\begin{align}
  \begin{aligned}
    &\partial_t w= \frac{1}{\rho} \partial_x \left( \frac{\partial_x w}{\rho} \right), \\
    &\frac{d}{dt} w(t,0) = - w(t,0), \hspace{0.2in} \frac{d}{dt} w(t,L) = - w(t,L), \hspace{0.2in} w(0,x) = X_0(x)- \Theta(x). 
  \end{aligned}
  \label{pseudoloc_cont_err}
\end{align}
\begin{assumption}\longthmtitle{Well-posedness of the
    pseudo-localization dynamics}
\label{assumption:smoothness_1d}
We assume that the pseudo-localization dynamics~\eqref{pseudoloc_cont}
(and~\eqref{pseudoloc_cont_err}) is well-posed, that the solution is
sufficiently smooth (at least~$\mathcal{C}^2$ in the spatial variable,
even as $t \rightarrow \infty$) and belong to the Sobolev
space~$H^1(M)$.
\end{assumption}
\begin{lemma}\longthmtitle{Pointwise convergence to diffeomorphism}
  Under Assumption~\ref{assumption:smoothness_1d}, on the
  well-posedness of the pseudo-localization dynamics, and Assumption~\ref{assumption:rho} 
  on the boundedness of~$\rho$, the solutions
  to PDE~\eqref{pseudoloc_cont} converge pointwise to the CDF
  diffeomorphism~$\Theta$ defined in \eqref{eq:theta}, as~$t
  \rightarrow \infty$, for all~$C^2$ initial conditions~$X_0$.
\end{lemma}
In this case, the swarm is stationary, which implies that the distribution~$\rho$ is fixed (and so is its support~$\bar{M}$),
and the uniform boundedness assumption~\ref{assumption:rho} simply becomes a boundedness
assumption.
\begin{proof}
  We prove that the solutions to the PDE~\eqref{pseudoloc_cont}
  converge pointwise to the diffeomorphism~$\Theta$ by showing that $w
  \rightarrow 0$, as $t \rightarrow \infty$, pointwise
  for~\eqref{pseudoloc_cont_err}.  For this, we consider a functional
  $V$, given by (integrations are with respect to the Lebesgue
  measure):
\begin{align*}
  V = \frac{1}{2} \int_{M} \rho |w|^2 + \frac{1}{2} \int_{M}
  \frac{1}{\rho}|\partial_x w|^2.
\end{align*}
The time derivative $\dot{V}$ is given by:
\begin{align*}
  \dot{V} = \int_{M} \rho w (\partial_t w) + \int_{M}
  \frac{1}{\rho} (\partial_x w) (\partial_t \partial_x w).
\end{align*}
Here, replace $\partial_t w$ in the first integral with
the dynamics in~\eqref{pseudoloc_cont_err}, and then use
$\partial_t \partial_x = \partial_x \partial_t$ in the second integral
together with the Divergence Theorem in Lemma~\ref{le:dt}. We obtain:
\begin{align*}
  \dot{V} &= \int_{M} w \partial_x \left( \frac{\partial_x
      w}{\rho} \right) - \int_{M} \partial_x 
  \left( \frac{\partial_x w}{\rho} \right) 
  \partial_t w + \frac{\partial_x w}{\rho} 
  \partial_t w \bigg|_L - \frac{\partial_x w}{\rho} \partial_t w \bigg|_0 \\
  &= - \int_{M} \frac{1}{\rho} \left| \partial_x w \right|^2 -
  \int_{M} \frac{1}{\rho} \left| \partial_x \left( \frac{\partial_x
        w}{\rho} \right) \right|^2 + \frac{w + \partial_t
    w}{\rho} \partial_x w \bigg|_L - \frac{w + \partial_t
    w}{\rho} \partial_x w \bigg|_0.
\end{align*}
(After the second equal sign, apply again the Divergence Theorem on
the first integral of the previous line, and replace $\partial_t w$
from~\eqref{pseudoloc_cont_err}.)  Substituting
from~\eqref{pseudoloc_cont_err}, we have:
\begin{align*}
  \dot{V} = - \int_{M} \frac{1}{\rho} \left| \partial_x w
  \right|^2 - \int_{M} \frac{1}{\rho} \left| \partial_x \left(
      \frac{\partial_x w}{\rho} \right) \right|^2.
\end{align*}
Clearly, $\dot{V} \leq 0$, and~$w(t,\cdot) \in H^1(M)$, for
all~$t$. Moreover, since~$V(t) \leq V(0)$ and since~$\rho$ is uniformly bounded according to Assumption~\ref{assumption:rho}, 
we have that~$w(t, \cdot)$ is bounded in~$H^1(M)$.
Moreover, by the Rellich-Kondrachov Theorem of
Lemma~\ref{lemma:RKCT}, $H^1(M)$ is compactly contained in~$L^2(M)$.
Then it follows that the solutions~$w(t, \cdot)$ are precompact.
Thus, by the LaSalle Invariance Principle of
Lemma~\ref{lemma:LaSalle_inv}, the solution
to~\eqref{pseudoloc_cont_err} converges in~$L^2$-norm to the largest
invariant subset of~$\dot{V}^{-1} (0)$. Note that $\dot{V} = 0$
implies $\int_M \frac{1}{\rho} |\partial_x w|^2= 0$. Thus, 
$\lim_{t \rightarrow \infty} \int_M \frac{1}{\rho} |\partial_x w|^2=
0$.  Since~$\rho$ is bounded ($\sup \rho < \infty$), we have $\lim_{t
  \rightarrow \infty} \frac{1}{\sup \rho} \int_M |\partial_x w|^2 \leq
\lim_{t \rightarrow \infty} \int_M \frac{1}{\rho} |\partial_x w|^2=
0$, which implies $\lim_{t \rightarrow \infty} \int_M |\partial_x w|^2
= \lim_{t \rightarrow \infty} \| \partial_x w \|_{L^2(M)}^2 = 0$.
%
Now, $\lim_{t \rightarrow
  \infty} |w(t,x)| = \lim_{t \rightarrow \infty} |w(t,0) +
\int_0^x \partial_x w(t,\cdot)| \leq \lim_{t \rightarrow \infty}
|w(t,0)| + \int_0^x |\partial_x w(t,\cdot)| \leq \lim_{t \rightarrow
  \infty} |w(t,0)| + \sqrt{L(t)} \| \partial_x w(t,\cdot) \|_{L^2(M)} = 0$ (since $ \lim_{t \rightarrow \infty} w(t,0) = 0$
  and  $ \lim_{t \rightarrow \infty} \| \partial_x w(t,\cdot) \|_{L^2(M)} = 0$). 
 Thus, $\lim_{t \rightarrow \infty} w(t,x) = 0$, for all $x \in M$.
Therefore, the solutions to~\eqref{pseudoloc_cont_err} converge to
$w\equiv0$ pointwise, as~$t \rightarrow \infty$, from any smooth
initial~$w_0 = X_0 - \Theta$.
\end{proof}

We now have that the solution to the pseudo-localization dynamics converges
to the diffeomorphism $\Theta$ in the stationary case. For the dynamic
case, we modify~\eqref{pseudoloc_cont} to account for agent
motion. Let $X: \real \times \real \rightarrow \real$ be supported
on~$\bar{M}(t) = [0,L(t)]$ for all~$t \geq 0$. Using the relation
$\frac{dX}{dt} = \partial_t X + v \partial_x X$, where $v$ is the
velocity field on the spatial domain, we consider:
\begin{align}
	\begin{aligned}
          &\partial_t X= \frac{1}{\rho} \partial_x \left( \frac{\partial_x X}{\rho} \right) - v \partial_x X, \\
          &X(t,0) = 0, \hspace{0.2in} X(t,L(t)) = \beta(t), \hspace{0.2in} X(0,x) = X_0(x).
	\end{aligned}
	\label{pseudoloc_cont_dynamic}
\end{align}
In the dynamic case, and w.l.o.g.~we have set~$\alpha(t) = 0$ for
all~$t \geq 0$, for simplicity.  We will use the above PDE system in
the design of the distributed motion control law, redesigning the
boundary control~$\beta$ to achieve convergence of the entire system.
We now discretize~\eqref{pseudoloc_cont_dynamic} to obtain a
distributed pseudo-localization algorithm.  Let $X_i(t) = X(t, x_i)$,
where $x_i \in \bar{M}(t)$ is the position of the~$\supscr{i}{th}$ agent. We
identify the agent~$i$ with its desired coordinate in the unit
interval at time~$t$, i.e., $\Theta(t, x) = \theta \in [0,1]$, where
$\Theta(t,x) = \int_{0}^{x} \rho(t,\bar{x}) d \bar{x}$ from
\eqref{eq:theta}, which now shows the time dependency of $\rho$. In
this way, $\rho(t, x) = \partial_x \Theta (t, x)$. It follows that
$\partial_x (\cdot)= \partial_{\theta}(\cdot) \partial_{x} \theta
= \partial_{\theta}(\cdot) \rho$. Therefore, $\frac{1}{\rho}\partial_x
(\cdot) = \partial_{\theta}(\cdot)$.  From
\eqref{pseudoloc_cont_dynamic}, we have:
\begin{align}
	\begin{aligned}
          \frac{dX}{dt} = \partial_t X +
          v \partial_x X = \frac{1}{\rho} 
          \partial_x \left( \frac{\partial_x X}{\rho} \right)
          = \partial_{\theta} \left( \partial_{\theta} X \right) =
          \frac{\partial^2 X}{\partial \theta^2}.
	\end{aligned}
	\label{pseudoloc_cont_discretization}
\end{align}
Now, we discretize~\eqref{pseudoloc_cont_discretization} with the
consistent finite differences $\frac{dX}{dt} \approx \frac{X_i(t+1) -
  X_i(t)}{\Delta t}$ and $\frac{\partial^2 X}{\partial \theta^2}
\approx \frac{X_{i+1} - 2 X_i + X_{i-1}}{(\Delta \theta)^2}$ (that is,
we have that $\lim_{\Delta t \rightarrow 0} \frac{X_i(t+1) -
  X_i(t)}{\Delta t} = \frac{dX}{dt}$ and that $\lim_{\Delta \theta
  \rightarrow 0} \frac{X_{i+1} - 2 X_i + X_{i-1}}{(\Delta \theta)^2}=
$ $\frac{\partial^2 X}{\partial \theta^2}$). Now, with the choice $3
\Delta t = (\Delta \theta)^2$, and from
\eqref{pseudoloc_cont_dynamic}, we obtain for $i \in \mathcal{S}
\setminus \left\lbrace l,r \right\rbrace$:
\begin{align}
	\begin{aligned}	
		&X_i(t+1) = \frac{1}{3}  \left( X_{i-1}(t) +  X_i(t) + X_{i+1}(t) \right),\\
		&X_l(t) = 0, \hspace{0.2in} X_r(t) = \beta(t), \hspace{0.2in} X_i(0) = {X_0}_i.
	\end{aligned} 
	\label{pseudoloc_discrete}
\end{align}
Equation \eqref{pseudoloc_discrete} is the discrete
pseudo-localization algorithm to be implemented synchronously by the
agents in the swarm, starting from any initial condition~$X_0$. The
leftmost agent holds its value at zero while the rightmost agent
implements the boundary control~$\beta$. In the following section we
analyze its behavior together with that of the dynamics on  $\rho$.
\subsection{Distributed density control law and analysis}
\label{subsec:distributed control}
In this subsection, we propose a distributed feedback control law to
achieve $\rho \rightarrow \rho^*$ and $w \rightarrow 0$, as $t
\rightarrow \infty$, through a distributed control input $v$ and a
boundary control $\beta$. 
We refer the reader to \cite{MK-AS:08} for an overview of
Lyapunov-based methods for stability analysis of PDE systems.

From~\eqref{eqn_continuity} and~\eqref{pseudoloc_cont_dynamic}, we
have the dynamics:
\begin{align}
	\begin{aligned}
          &\partial_t \rho = -\partial_x (\rho v), \\
          &\partial_t X = \frac{1}{\rho} \partial_x \left( \frac{\partial_x X}{\rho} \right) - v \partial_x X, \\
          &X(t, 0) = 0, \hspace{0.2in} X(t, L(t)) = \beta(t), \hspace{0.2in} X(0,x) = X_0(x).
	\end{aligned}
	\label{eq:1D_system_dynamics}
\end{align}
This realizes the feedback interconnection of
Figure~\ref{fig:interconnection}.
\begin{assumption}\longthmtitle{Well-posedness of the full PDE system}
\label{as:wellp}
  We assume that~\eqref{eq:1D_system_dynamics} is well posed, and that
  the solutions~$( \rho(t, \cdot), X(t,\cdot) )$ are sufficiently
  smooth (both in~$t$ and~$x \in [0,L(t)]$), satisfy Assumption~\ref{assumption:rho} on the 
  uniform boundedness of~$\rho$ and~$\partial_x \rho$, and are bounded in the Sobolev space~$ H^1((0,1/d_l))$.
\end{assumption}
We also assume that the agent at position~$x$ at time~$t$ is able to
measure $\rho(t, x)$.  However, the agents in the swarm do not have
access to their positions, and therefore cannot access $\rho^*(x)$,
which could be used to construct a feedback law. To circumvent this
problem, we propose a scheme in which the agents use the position
identifier or pseudo-localization variable $X$ to compute $p^*\circ
X(t,x)$, using this as their dynamic set-point. The idea is to then
design a distributed control law and a boundary control law such that
$\rho \rightarrow p^*\circ X$ and $X \rightarrow \Theta^*$, as $t
\rightarrow \infty$, to obtain $\rho \rightarrow p^*\circ \Theta^* =
\rho^*$.  Recall that the function $p^*$ is computed offline and is
broadcasted to the agents prior to the beginning of the
self-organization process, and that~$p^*$ is assumed to be a Lipschitz
function.
Consider the distributed control law, defined as follows for all time
$t$:
\begin{align}
\begin{aligned}
  &v(t,0) = 0, \hspace{0.2in} \partial_x v = (\rho - p^* \circ X) - \frac{\partial_X
    p^*}{\rho(\rho + p^* \circ X)}\partial_x \left( \frac{\partial_x
      X}{\rho} \right),
\end{aligned}
	\label{distributed_control_law}
\end{align}
together with the boundary control law:
\begin{align}
	\begin{aligned}
          &X(t,0) = 0, \hspace{0.2in} \beta_t = k \left( 2 - \beta(t) - \frac{X_x}{\rho}\bigg|_{L(t)} \right).
	\end{aligned}
	\label{boundary_control_law}
\end{align}
We remark again that the agents implementing the control
laws~\eqref{distributed_control_law} and~\eqref{boundary_control_law}
do not require position information, because for the agent at
position~$x$ at time~$t$, $\rho(t,x)$ is a measurement, $X(t,x)$ is
the pseudo-localization variable, through which $p^* \circ X(t,x)$ can
be computed.
\begin{theorem}\longthmtitle{Convergence of solutions}
  Under the well-posedness Assumption~\ref{as:wellp}, the
  solutions~$(\rho(t,\cdot), X(t,\cdot))$
  to~\eqref{eq:1D_system_dynamics}, under the control laws
  \eqref{distributed_control_law} and \eqref{boundary_control_law},
  converge to~$(\rho^*, \Theta^*)$, $\rho \rightarrow \rho^*$ in $L^2-norm$ and $X
  \rightarrow \Theta^*$ pointwise as~$t \rightarrow \infty$, from any
  smooth initial condition~$(\rho_0$, $ X_0)$.
\end{theorem}
\begin{proof}
  Consider the candidate control Lyapunov functional $V$: 
  \begin{align*}
    V = \frac{1}{2} \int_{0}^{L(t)} |\rho - p^*\circ X |^2 dx +
    \frac{1}{2} \int_{0}^{L(t)} \rho |w|^2 dx +
    \frac{1}{2} |w(L(t))|^2.
\end{align*}
Taking the time derivative of $V$ along the dynamics
\eqref{eq:1D_system_dynamics}, using Lemma~\ref{lemma:Leibniz_rule} on
the Leibniz integral rule, and applying Corollary~\ref{lemma:time_der}
on the derivative of energy functionals, we obtain:
\small
\begin{align*}
  \dot{V} =& \int_{0}^{L(t)} (\rho - p^*\circ X)\left(\frac{d
      \rho}{dt} - \frac{d (p^*\circ X) }{dt} \right) dx +
  \frac{1}{2}\int_0^{L(t)}|\rho - p^*\circ X|^2 \partial_x v~dx
  \\
  &+ \int_0^{L(t)} \rho w \partial_t w~dx  + \frac{1}{2} \int_0^{L(t)} (\partial_t \rho) |w|^2~dx + \frac{1}{2} \rho |w|^2 v \bigg|_{0}^{L(t)}  
    + w(L)\frac{dw(L(t))}{dt}.
\end{align*}
\normalsize
Now, $\frac{d \rho}{dt} = \partial_t \rho + v \partial_x \rho = -
\rho \partial_x v$ (since $\partial_t \rho = - \partial_x(\rho v)$,
from~\eqref{eq:1D_system_dynamics}), and $\partial_t w =
\frac{1}{\rho} \partial_x \left( \frac{\partial_x w}{\rho} \right) -
v \partial_x w$.
Thus, we obtain:
\small
\begin{align*}
  \dot{V} = & \int_{0}^{L(t)} (\rho - p^*\circ X)
  \left[-\rho \partial_x v - \partial_X p^* \frac{1}{\rho} \partial_x
    \left( \frac{\partial_x X}{\rho} \right) \right] dx + \frac{1}{2}\int_0^{L(t)}|\rho - p^*\circ X|^2 \partial_x v~dx \\
    &+ \int_0^{L(t)} w \partial_x \left( \frac{\partial_x w}{\rho} \right)~dx 
    -  \int_0^{L(t)} \rho v w \partial_x w~dx  - \frac{1}{2} \int_0^{L(t)} \partial_x(\rho v) |w|^2~dx + \frac{1}{2} \rho |w|^2 v \bigg|_{0}^{L(t)} \\
  & + w(L)\frac{dw(L(t))}{dt}.
\end{align*}
\normalsize
Now, using the above equation, applying the Divergence
theorem~\eqref{eqn:divergence_thm} (integration by parts) 
and rearranging the terms, we
obtain:
\begin{align*}
  \dot{V} =& -\frac{1}{2} \int_{0}^{L(t)} (\rho - p^*\circ X)
  \left[(\rho + p^*\circ X) (\partial_x v) + \frac{\partial_X
      p^*}{\rho}  
    \partial_x \left( \frac{\partial_x X}{\rho} \right) \right]~dx \\
  & + \frac{w \partial_x w}{\rho} \bigg|_{0}^{L(t)} - \int_{0}^{L(t)} \frac{|\partial_x w|^2}{\rho}~dx
   -  \int_0^{L(t)} \rho v w \partial_x w~dx - \frac{1}{2} \rho v |w|^2 \bigg|_{0}^{L(t)} \\
   & + \int_0^{L(t)} \rho v w \partial_x w~dx  + \frac{1}{2} \rho |w|^2 v \bigg|_{0}^{L(t)} + w(L)\frac{dw(L(t))}{dt}.
\end{align*}
%
Since $w(0) = 0$, the above
equation reduces to:
\begin{align*}
	\dot{V} =& -\frac{1}{2} \int_{0}^{L(t)} (\rho - p^*\circ X)
        \left[(\rho + p^*\circ X) (\partial_x v) + \frac{\partial_X
            p^*}{\rho}  
          \partial_x \left( \frac{\partial_x X}{\rho} \right) \right]~dx \\
        & - \int_{0}^{L(t)} \frac{|\partial_x w|^2}{\rho}~dx +  w(L(t)) \left( \frac{d}{dt} w(L(t)) + \frac{\partial_x w}{\rho} \right) .
\end{align*}
From~\eqref{distributed_control_law} and~\eqref{boundary_control_law},
we have $\partial_x v = (\rho - p^* \circ X) - \frac{\partial_X
  p^*}{\rho(\rho + p^* \circ X)}\partial_x \left( \frac{\partial_x
    X}{\rho} \right)$, and $\frac{dw}{dt}\bigg|_{L(t)} =
- \left( \frac{\partial_x w}{\rho} + k w \right)\bigg|_{L(t)}$, and we
obtain:
\small
\begin{align}
\begin{aligned}
  \dot{V} = &-\frac{1}{2} \int_{0}^{L(t)} (\rho + p^*\circ X)|\rho -
  p^*\circ X|^2 dx - \int_{0}^{L(t)} \frac{|\partial_x w|^2}{\rho}~dx 
    - k \left| w(L(t)) \right|^2.
    \end{aligned}
    \label{eq:dotV_1D}
\end{align}
\normalsize
Clearly, $\dot{V} \leq 0$, and~$\rho(t,\cdot), w(t,.) \in H^1((0,
1/d_l))$, for all~$t$. By Lemma~\ref{lemma:RKCT}, the
Rellich-Kondrachov Compactness Theorem, the space $H^1((0, 1/d_l))$ 
is compactly contained in~$L^2((0, 1/d_l))$, and the bounded
solutions (by Assumption~\ref{as:wellp}) in $H^1((0, 1/d_l))$ are then
precompact in~$L^2((0, 1/d_l))$. Moreover, the set of $( \rho, X )$ satisfying
Assumption~\ref{as:wellp} is dense in~$L^2((0, 1/d_l))$.
 Then, by the LaSalle Invariance Principle,
Lemma~\ref{lemma:LaSalle_inv}, we have that the solutions to
\eqref{eq:1D_system_dynamics} converge in the~$L^2$-norm to the
largest invariant subset of~$\dot{V}^{-1} ( 0 )$. This implies that:
\begin{align*}
  &\lim_{t \rightarrow \infty}
  \| \rho(t, \cdot) - p^* \circ X(t,\cdot) \|_{L^2((0,L(t)))} = 0, \\
  &\lim_{t \rightarrow \infty} \| \frac{\partial_x w}{\rho}
  \|_{L^2((0,L(t)), \rho)}
  = 0, \hspace{0.2in}   \lim_{t \rightarrow \infty} w(t,L(t)) = 0.
\end{align*}
Thus, we have:
\begin{align*}
	\lim_{t \rightarrow \infty} \left \| \frac{\partial_x w}{\rho} \right \|_{L^2((0,L(t)), \rho)} = 0 \hspace{0.2in} \Rightarrow & \lim_{t \rightarrow \infty} \| \partial_x w \|_{L^2((0,L(t)))} = 0.
\end{align*}
Using the Poincar\'e-Wirtinger inequality,
Lemma~\ref{lemma:poincare_wirtinger}, again, we note that this implies
$ \lim_{t \rightarrow \infty} \| w - \int_0^{L(t)} w
\|_{L^2((0,L(t)))} = 0$.  We have $\lim_{t \rightarrow \infty} |
\int_0^{L(t)} w | = | \int_0^{L(t)} \int_0^x \partial_x w | \leq
L(t)^{3/2} \| \partial_x w\|_{L^2((0,L(t)))} = 0$, which implies that
$\lim_{t \rightarrow \infty} \int_0^{L(t)} w = 0$ and therefore $
\lim_{t \rightarrow \infty} \| w \|_{L^2((0,L(t)))} = 0$.  Thus, we
get $\lim_{t \rightarrow \infty} \| w(t,\cdot) \|_{H^1((0,L(t)))} =
0$, or in other words, $w \rightarrow_{H^1} 0$. Now, $\lim_{t \rightarrow \infty} |w(t,x)| = \lim_{t
  \rightarrow \infty} |w(t,0) + \int_0^x \partial_x w(t,\cdot)| \leq
\lim_{t \rightarrow \infty} |w(t,0)| + \int_0^x |\partial_x
w(t,\cdot)| \leq \lim_{t \rightarrow \infty} |w(t,0)| + \sqrt{L(t)} \|
w(t,\cdot) \|_{H^1((0,L(t)))} = 0$, which implies that $w \rightarrow 0$
pointwise.  Given that~$w = X - \Theta$, we have~$\lim_{t \rightarrow
  \infty} X(t,\cdot) - \Theta(t, \cdot) = 0$.  Let~$\lim_{t
  \rightarrow \infty} L(t) = L$ and $\lim_{t \rightarrow \infty}
\Theta(t, \cdot) = \bar{\Theta}(\cdot)$, which implies that
$X \rightarrow \bar{\Theta}$ pointwise.

From the above, we have $\lim_{t \rightarrow \infty} \| \rho(t,
\cdot) - p^* \circ \bar{\Theta} \|_{L^2((0,L(t)))} = \lim_{t
  \rightarrow \infty} \| \rho(t, \cdot) - p^* \circ X(t, \cdot) + p^*
\circ X(t, \cdot) - p^* \circ \bar{\Theta} \|_{L^2((0,L(t)))} \leq
\lim_{t \rightarrow \infty} \| \rho(t, \cdot) - p^* \circ X(t, \cdot)
\|_{L^2((0,L(t)))} + \| p^* \circ X(t, \cdot) - p^* \circ \bar{\Theta}
\|_{L^2((0,L(t)))} = 0$ (this follows from the assumption that~$p^*$
is Lipschitz, since~$\| p^* \circ X - p^* \circ \bar{\Theta} \|_{L^2}
\leq c \| X - \bar{\Theta} \|_{L^2}$ for some Lipschitz
constant~$c$). Thus, we have $\rho \rightarrow_{L^2} p^* \circ
\bar{\Theta}$.

Now, we are interested in the limit density distribution $\bar{\rho} =
p^* \circ \bar{\Theta}$, and by the definition of~$\bar{\Theta}$ we
have~$\bar{\Theta}(x) = \int_0^x \bar{\rho}$.  We now prove that this
limit $(\bar{\rho}, \bar{\Theta})$ is unique, and that $(\bar{\rho},
\bar{\Theta}) = (\rho^*, \Theta^*)$.  From the definition of
$\bar{\Theta}$, we get $ \frac{d\bar{\Theta}}{dx}(x) = \bar{\rho}(x) =
p^*(\bar{\Theta}(x))>0$, $\forall \bar{\Theta}(x) \in [0,1]$. We
therefore have:
\begin{align*}
	x = \int_0^{\bar{\Theta}(x)} \left( p^*(\theta) \right)^{-1} d\theta.
\end{align*}
Recall from the definition of $p^*$ and \eqref{eq:theta} that $p^*
\circ \Theta^*(x) = \rho^*(x)$, and $\frac{d}{dx}\Theta^*(x) = \rho^*(x) = p^*
\circ \Theta^*(x)$, which implies that $ \frac{d\Theta^*}{dx} =
p^*(\theta^*)>0$, where $\theta^* = \Theta^*(x)$. Therefore:
\begin{align*}
	x = \int_0^{\Theta^*(x)} \left( p^*(\theta) \right)^{-1} d\theta.
\end{align*}
From the above two equations, we get:
\begin{align*}
  \int_0^{\bar{\Theta}(x)} \left( {p^*}(\theta)\right)^{-1} d\theta =
  \int_0^{\Theta^*(x)} \left( {p^*}(\theta)\right)^{-1} d\theta,
\end{align*}
for all~$x$, and since $p^*$ is strictly positive, it implies that
$\bar{\Theta} = \Theta^*$, and we obtain $\bar{\rho} = p^*\circ
\bar{\Theta} = p^* \circ \Theta^* = \rho^*$.
And we know that
$\rho \rightarrow_{L^2} p^* \circ \bar{\Theta} = p^* \circ \Theta^* =
\rho^*$.  In other words,~$\rho$ converges to~$\rho^*$ in the $L^2$
norm.  
\end{proof}
\subsubsection{Physical interpretation of the density control law}
For a physical interpretation of the control law, we first rewrite
some of the terms in a suitable form.  
From~\eqref{eq:1D_system_dynamics}, we know that:
\begin{align*}
  \frac{1}{\rho}\partial_x \left( \frac{\partial_x X}{\rho} \right) =
  \frac{\partial X}{\partial t} + v \partial_x X =\frac{dX}{dt}.
\end{align*}
The second term in the expression for $\partial_x v$ in the
law~\eqref{distributed_control_law} can thus be rewritten as:
\begin{align*}
  \frac{\partial_X p^*}{\rho (\rho + p^* \circ X)} \partial_x \left(
    \frac{\partial_x X}{\rho} \right) = \frac{1}{(\rho + p^* \circ
    X)}~ \partial_X p^* \frac{dX}{dt} = \frac{1}{(\rho + p^* \circ X)}
  \frac{d p^*}{dt}.
\end{align*}   
Now, from above and~\eqref{distributed_control_law}, we obtain:
\begin{align}
  v(t,x) = \int_0^x (\rho - p^* \circ X) - \int_0^x \frac{1}{(\rho +
    p^* \circ X)} \frac{d p^*}{dt}.
	\label{eq:velocity}
\end{align}
Equation~\eqref{eq:velocity} gives the velocity of the agent at~$x$ at
time~$t$.  Now, to interpret it, we first consider the case where the
pseudo-localization error is zero, that is, when $X = \Theta^*$.  This
would imply that $p^* \circ X = p^* \circ \Theta^* = \rho^*$,
$\frac{dX}{dt} = \frac{d\Theta^*}{dt} = 0$, and we obtain:
\begin{align}
	v(t,x) = \int_0^x (\rho - \rho^*).
	\label{eq:velocity_ideal}
\end{align}
The term $\int_0^x (\rho - \rho^*) = \int_0^x \rho - \int_0^x \rho^*$
is the difference between the number of agents in the interval $[0,x]$
and the desired number of agents in $[0,x]$. If the term is positive,
it implies that there are more than the desired number of agents in
$[0,x]$ and the control law essentially exerts a pressure on the agent
to move right thereby trying to reduce the concentration of agents in
the interval $[0,x]$, and, vice versa, when the term is negative.  This
eventually accomplishes the desired distribution of agents over a
given interval.  This would be the physical interpretation of the
control law for the case where the pseudo-localization error is zero
(that is, the agents have full information of their positions).

However, in the transient case when the agents do not possess full
information of their positions and are implementing the
pseudo-localization algorithm for that purpose, the control law
requires a correction term that accounts for the fact that the
transient pseudo coordinates~$X(t,x)$ cannot be completely relied
upon.  This is what the second term $ \int_0^x \frac{1}{(\rho + p^*
  \circ X)} \frac{d p^*}{dt}$ in~\eqref{eq:velocity} corrects
for. When this term is positive, that is, $\int_0^x \frac{1}{(\rho +
  p^* \circ X)} \frac{d p^*}{dt} > 0$, it roughly implies that the
``estimate" of the desired number of agents in the interval~$[0,x]$ is
increasing (indicating that an increase in the concentration of agents
in $[0,x]$ is desirable), and the term essentially reduces the
``rightward pressure" on the agent (note that this term will have a negative
contribution to the velocity~\eqref{eq:velocity}).
\subsection{Discrete implementation}
In this section, we present a scheme to compute~$p^*$ (the transformed
desired density profile) and a consistent discretization scheme for
the distributed control law. We follow that up with a discussion on
the convergence of the discretized system and a pseudo-code for the
implementation.
\subsubsection{On the computation of~$p^*$}
\label{sec:computingpstar_1D}
We now provide a method for computing $p^*$ from a given
$\rho^*$ via interpolation. Let the desired domain $M^* = [0,L^*]$ be
discretized uniformly to obtain $M^*_d = \lbrace 0= x_1, \ldots, x_m =
L^* \rbrace$ such that $x_j - x_{j-1} = h$ (constant step-size).  Note
that $m$ is the number of interpolation points, not equal to the
number of agents. The desired density~$\rho^* : [0,L^*] \rightarrow
\realpositive$ is known, and we compute the value of~$\rho^*$
on~$M^*_d$ to get $\rho^*(x_1, \ldots, x_m) = (\rho^*_1, \ldots,
\rho^*_m)$.  We also have $\Theta^*(x) = \int_0^x \rho^* d\mu$, for
all $x \in [0,L^*]$.  Now, computing the integral with respect to the
Dirac measure for the set~$M^*_d$,
we obtain $\Theta_d^*(x_1, \ldots, x_m) = (\theta^*_1, \ldots,
\theta^*_m)$, where $\theta^*_1 = 0$ and $\theta^*_k =
\frac{1}{2}\sum_{j=1}^{k} (\rho^*_{j-1} + \rho^*_j)h$, 
for $k=2, \ldots, m$ (note that $0 =
\theta^*_1 \leq \theta^*_2 \leq \ldots \leq \theta^*_m \leq 1$ and
$\lim_{h \rightarrow 0} \theta^*_m = \Theta^*(L^*) = 1$).  Now, the
value of the function~$p^*$ at any~$X \in [0,1]$ can be now obtained
from the relation $p^*(\theta^*_k) = \rho^*_k$, for $k = 1, \ldots,
m$, by an appropriate interpolation.
  \begin{figure}[H]
  \centering
  	\begin{tikzcd}
          &(\rho^*_1, \ldots, \rho^*_m)
          = p^*(\theta^*_1, \ldots, \theta^*_m) \\
          (x_1, \ldots, x_m) \arrow[ur, mapsto, shift left, end anchor
          = {[xshift = -1.5ex]}, "\rho^*"] \arrow[r, mapsto, end
          anchor = {[xshift = 6ex]}, "\Theta^*"] & \hspace{1cm}
          (\theta^*_1, \ldots, \theta^*_m) \arrow[u, Mapsto, shift
          right = 4ex, "p^*"]
  	\end{tikzcd}
  \end{figure}

\subsubsection{Discrete control law}
A discretized pseudo-localization algorithm is given
by~\eqref{pseudoloc_discrete}.  We now discretize
\eqref{distributed_control_law} to obtain an implementable control law
for a finite number of agents $i \in \Sc$, and a numerical simulation
of this law is later presented in Section~\ref{sec:simulation}.

Let $i \in \Sc\setminus \{l,r\}$.  First note that $\partial_x v =
(\partial_{\theta} v) \bigg|_{\theta = \Theta(x)} (\partial_x \Theta)
= (\partial_{\theta} v) \bigg|_{\theta = \Theta(x)} \rho$ (where~$v
\equiv v(\Theta(x))$). Using a consistent backward differencing
approximation, and recalling that $\Delta \theta = \epsilon$, we can
write:
\begin{align*}
  (\partial_x v)_i \approx \rho_i \frac{v_{i} - v_{i-1}}{ \Delta \theta} = \rho_i
  \frac{v_{i} - v_{i-1}}{\epsilon}, \quad i \in \Sc
\end{align*}
where $\rho_i$ is agent~$i$'s density measurement.

From Section~\ref{sec:pseudoloc_1D}, recall the consistent
finite-difference approximation:
\begin{align*}
	\frac{1}{\rho}
{\partial_x}\left(\frac{\partial_xX}{\rho}\right)_i
\approx \frac{1}{\epsilon^2} (X_{i-1} - 2 X_i +
  X_{i+1}).
\end{align*} 
   With $\kappa = \frac{1}{2
  \epsilon}$, from \eqref{distributed_control_law} and the above
equation, we obtain the law for agent~$i$ as:
\small
\begin{align}
	\begin{aligned}
          v_i = v_{i-1} + \frac{\rho_i - p^*(X_i)}{2 \kappa \rho_i} -
          \frac{2 \kappa}{\rho_i(\rho_i + p^*(X_i))}
           \cdot \frac{p^*(X_{i+1}) - p^*(X_{i-1})}{X_{i+1} -
              X_{i-1}} \cdot (X_{i-1} - 2 X_i +
  X_{i+1}) 
	\end{aligned}
	\label{eq:disc_vel_update}
\end{align}
\normalsize
with~$v_l = 0$. The computation in $v$ can be implemented by
  propagating from the leftmost agent to the
rightmost agent along a line graph~$\mathcal{G}_{line}$ (with message receipt
acknowledgment). Note that this propagation can alternatively be
formulated by each agent averaging appropriate variables with left and
right neighbors, which will result in a process similar to a finite-time
consensus algorithm.
Now, the boundary control~\eqref{boundary_control_law} is discretized
(with~$\partial_t \beta \approx \frac{\beta(t+1) - \beta(t)}{\Delta
  t}$), with the choice~$k = \frac{1}{\epsilon}$ to:
 \small
\begin{align}
\begin{aligned}
  \beta(t+1) &= \beta(t) + k\Delta t (2 - \beta(t) - 2 \kappa \left(
    \beta(t) - X_{r-1}(t))
  \right) 
  = \frac{4 - 2 \epsilon}{3}\beta(t) + \frac{1}{3} X_{r-1}(t) 
\end{aligned}
\label{eq:disc_bdry_ctrl}
\end{align}
\normalsize
\subsubsection{On the convergence of the discrete system}
The discretized pseudo-localization
algorithm~\eqref{pseudoloc_discrete} with the boundary control
law~\eqref{boundary_control_law}, can be rewritten as:
\begin{align}
	X(t+1) = X(t) - \frac{1}{3} L X(t) + u(t),
	\label{eq:disc_pseudoloc_vectorform}
      \end{align} where~$X(t) = (X_l(t), \ldots, X_r(t))$, $L$ is the
      Laplacian of the line graph~$\mathcal{G}_{line}$ and the input
      $u(t) = \left( 0, \ldots, 0, \frac{\epsilon}{3} (2 - \beta(t) )
      \right)$. This discretized system is stable and we thereby have
      that the discretized pseudo-localization algorithm is consistent
      and stable.  Thus, by the Lax Equivalence Theorem~\cite{GDS:85},
      the solution of~\eqref{eq:disc_pseudoloc_vectorform} converges
      to the solution of~\eqref{pseudoloc_cont_dynamic} with the
      boundary control~\eqref{boundary_control_law} as~$N \rightarrow
      \infty$.
      Due to the nonlinear nature of the discrete implementation of
      the equation in $\rho$, we are only certain that we have a
      consistent discrete implementation in this case (no similar
      convergence theorem exists for discrete approximations of
      nonlinear PDEs.)
\begin{algorithm}[H]
\footnotesize
\label{alg:1D_implementation}
\begin{algorithmic}[1]
\caption{Self-organization algorithm for 1D environments}
	\State \textbf{Input:} $\rho^*$, $K$ (number of iterations), $\Delta t$ (time step)
        \State \textbf{Requires:} 
	\State $\quad$Offline computation of~$p^*$ as outlined in Section~\ref{sec:computingpstar_1D}
	\State $\quad$Initialization $X_i (0) = {X_0}_i$, $v_i = 0$ 
        \State $\quad$Leftmost and rightmost agents, $l$, $r$, resp., are
        aware they are at boundary
	\For{$k:=1$ to $K$} 
		\If{$i = l$ }
		\State agent $l$ holds onto $X_l(k) = 0$ and $v_l(k) = 0$
		\ElsIf{agent $i \in \lbrace l+1 , \ldots, r-1 \rbrace$ }
		\State agent $i$ receives $X_{i-1}(k)$ and $X_{i+1}(k)$ from its left and right neighbors
		\State agent $i$ implements the update~\eqref{pseudoloc_discrete} 
		\ElsIf{$i = r$}
		\State agent $r$ receives $X_{r-1}(k)$ from its left neighbor
		\State agent~$r$ implements the update~\eqref{eq:disc_bdry_ctrl}
		\EndIf
			\For{$i:= l$ to $r$}
				\State agent~$i$ computes velocity $v_i$ from \eqref{eq:disc_vel_update} 
			\EndFor
			\State agent~$i$ moves to 
                           $x_i(k+1) = x_i(k) + v_i(k) \Delta t$
	\EndFor
\end{algorithmic}
\end{algorithm}
%
\section{Self-organization in two dimensions}
\label{sec:self-organization in 2D}
In this section, we present the two-dimensional self-organization
problem. Although our approach to the ~$2$D problem is fundamentally
similar to the ~$1$D case, we encounter a problem in the
two-dimensional case that did not require consideration in one
dimension, and it is the need to control the shape of the spatial
domain in which the agents are distributed. We overcome this problem
by controlling the shape of the domain with the agents on the
boundary, while controlling the density distribution of the agents in
the interior.

Let $M : \real \rightrightarrows \real^2$ be a smooth one-parameter
family of bounded open subsets of~$\real^2$, such that~$\bar{M}(t)$ is
the spatial domain in which the agents are distributed at time~$t \geq
0$. Let $\rho : \real \times \real^2 \rightarrow \real_{\geq 0}$ be
the spatial density function with support~$\bar{M}(t)$ for all $t \geq
0$; that is, $\rho(t,x) > 0$, $\forall\,x \in \bar{M}(t)$, and $t \ge
0$. Without loss of generality, we shift the origin to a point on the
boundary of the family of domains, such that $(0,0) \in \partial
M(t)$, for all~$t$. Let $\rho^*: M^* \rightarrow \realpositive$ be the
desired density distribution, where~$M^*$ is the target spatial
domain. From here on, we view~$\bar{M}$ as a one-parameter family of
compact $2$-submanifolds with boundary of~$\real^2$. Just as in the
$1$D case, the agents do no have access to their positions but know
the true~$x$- and~$y$-directions.

In what follows we present our strategy to solve this problem, which
we divide into three stages for simplicity of presentation and
analysis. In the first stage, the agents converge to the target
spatial domain~$M^*$ with the boundary agents controlling the shape of
the domain. In stage two, the agents implement the pseudo-localization
algorithm to compute the coordinate transformation. In the third
stage, the boundary agents remain stationary and the agents in the
interior converge to the desired density distribution. This
simplification is performed under the assumption that, once the agents
have localized themselves at a given time, they can accurately update
this information by integrating their (noiseless) velocity inputs.
Noisy measurements would require that these phases are rerun with some
frequency; e.g.~using fast and slow time scales as described in
Section~\ref{sec:conceptual}.

\subsection{Pseudo-localization algorithm for boundary agents}
\label{sec:boundary_localization}

To begin with, we propose a pseudo-localization algorithm for the
boundary agents which allows for their control in the first stage.  To
do this, we assume that the agents have a boundary detection
capability (can approximate the normal to the boundary), the ability
to communicate with neighbors immediately on either side along the
boundary curve, and can measure the density of boundary agents.

Let $M_0 \subset \real^2$ be a compact $2$-manifold with boundary
$\partial M_0$ and let $(0,0) \in \partial M_0$.  To localize
themselves, the agents on $\partial M_0$ implement the
distributed~$1$D pseudo-localization algorithm presented in
Section~\ref{sec:pseudoloc_1D}.  This yields a parametrization of the
boundary $\Gamma: \partial M_0 \rightarrow [0,1)$, with $\Gamma(0,0) =
0$, such that the closed curve which is the boundary~$\partial M_0$ is
identified with the interval $[0,1)$.
We have that, for $\gamma \in [0,1)$, $\Gamma^{-1}(\gamma)
\in \partial M_0$. For $\gamma \in [0,1)$, let $s(\gamma)$ be the arc
length of the curve $\partial M_0$ from the origin, such that $s(0) =
0$ and $\lim_{\gamma \rightarrow 1} s(\gamma) = l$. We assume that the
boundary agents have access to the unit outward normal
$\mathbf{n}(\gamma)$ to the boundary, and thus the unit tangent
$\mathbf{s}(\gamma)$.

Let $q: [0,l) \rightarrow \realpositive$ denote the normalized density
of agents on the boundary, such that we have $\int_0^l q(s) ds =
1$. Now the 1D pseudo-localization algorithm of
Section~\ref{sec:pseudoloc_1D} serves to provide a 2D boundary
pseudo-localization as follows. Note that $\frac{ds}{d \gamma} =
\frac{1}{q(\gamma)}$, and $(dx,dy) = \mathbf{s} ds$, which implies $
(dx,dy) = \frac{1}{q(\gamma)} \mathbf{s}(\gamma) d \gamma$. Therefore,
we get the position of the boundary agent at $\gamma$,
$(x(\gamma),y(\gamma))$, as $(x(\gamma),y(\gamma)) = \int_{0}^{\gamma}
\frac{1}{q(\bar{\gamma})} \mathbf{s}(\bar{\gamma}) d \bar{\gamma}$,
and the arc-length $s(\gamma) = \int_{0}^{\gamma}
\frac{1}{q(\bar{\gamma})} d \bar{\gamma}$, which is discretized by a
consistent scheme to obtain:
\begin{align}
  (x_i, y_i) = \frac{1}{2} \Delta \gamma \sum_{k=0}^{i-1} \left(
    \frac{\mathbf{s}_{k}}{q_{k}} + \frac{\mathbf{s}_{k+1}}{q_{k+1}}
  \right), \qquad \text{for}~i \in \partial M_0,
	\label{eq:2D_boundary_pseudoloc_ss}
\end{align}
 and we recall that the agents have access to $q$ and
$\mathbf{s}$. 
The computation of $(x_i, y_i)$ can be implemented by 
  propagating from the agent with $\gamma_i = 0$ along the boundary
 agents in the direction as $\gamma_i \rightarrow 1$, 
   along a line graph~$\subscr{\mathcal{G}}{line}$ (with message receipt
acknowledgment). Note that this propagation can alternatively be
formulated by each agent averaging appropriate variables with left and
right neighbors, which will result in a process similar to a finite-time
consensus algorithm.

This way, the boundary agents are localized at time
$t=0$, and they update their position estimates using their
velocities, for $t \geq 0$.

\subsection{Pseudo-localization algorithm in two dimensions}
\label{sec:2D_pseudoloc}

In this subsection, we present the pseudo-localization algorithm for
the agents in the interior of the spatial domain. We first describe
the idea of the coordinate transformation (diffeomorphism) we employ
and construct a PDE that converges asymptotically to this
diffeomorphism. We then discretize the PDE to obtain the distributed
pseudo-localization algorithm.

The main idea is to employ harmonic maps to construct a coordinate
transformation or diffeomorphism from the spatial domain of the swarm
onto the unit disk. We begin the construction with the static case,
where the agents are stationary. Let $M \subseteq \real^2$ be a
compact, static $2$-manifold with boundary and $N = \{ (x,y) \in
\real^2\,|\, (x-1)^2 + y^2 \leq 1 \}$ be the unit disk. The manifolds
$M$ and $N$ are both equipped with a Euclidean metric $g = h = \delta$.

First, we define a mapping for the boundary of $M$. Let
$\Gamma: \partial M \rightarrow [0,1)$ be a parametrization of the
boundary of~$M$, as outlined in
Section~\ref{sec:boundary_localization}. Let $\xi : \bar{M}
\rightarrow N$ be any diffeomorphism that takes the following form on
the boundary of~$M$:
\begin{align}
  \xi(\Gamma^{-1}(\gamma)) = (1-\cos(2 \pi \gamma), \sin(2 \pi
  \gamma)), \qquad  \gamma \in [0,1),
  \label{eq:boundary_maponto_circle}
\end{align} 
and we know that $\Gamma^{-1}[0,1) = \partial M$. 

Now, from Lemma~\ref{lemma:harmonic_diffeomorphism}, on harmonic
diffeomorphisms, there is a unique harmonic diffeomorphism,
$\Psi : M \rightarrow N$, such that $\Psi = \xi$ on
$\partial M$. We know that, by definition, the mapping $\Psi =
(\psi_1, \psi_2)$ satisfies:
\begin{align}
	\begin{aligned}
	\begin{cases} 
          \Delta \psi_1 = 0, \\
          \Delta \psi_2 = 0,
    \end{cases}
    &\text{for } \mathbf{r} \in \mathring{M}, \\
    \Psi = \xi, \hspace{0.15in} &\text{on } \partial M,
	\end{aligned}
    \label{eq:harmonic_diffeo}
\end{align} 
where $\Delta$ is the Laplace operator. Let~$\Psi^*$
 be the corresponding map from the target domain~$M^*$ to the unit disk~$N$. 
 Now, we define a
function $p^*: N \rightarrow \real_{>0}$ by $p^* = \rho^* \circ
({\Psi^*})^{-1}$, the image of the desired spatial density
distribution on the unit disk, which is computed offline and is
broadcasted to the agents prior to the beginning of the
self-organization process. We later use~$p^*$ to derive the
distributed control law which the agents implement.
  \begin{figure}[H]
  \centering
  	\begin{tikzcd}
  		&\rho^*(\mathbf{r}) = p^*(\Psi^*(\mathbf{r})) \\
  		\mathbf{r} \in M^* \arrow[ur, mapsto, shift left, end anchor = {[xshift = -1.5ex]}, "\rho^*"] \arrow[r, mapsto, end anchor = {[xshift = 6ex]}, "\Psi^*"]  & \hspace{1cm}  \Psi^*(\mathbf{r}) \in N \arrow[u, mapsto, shift right = 4ex, "p^*"]
  	\end{tikzcd}
  \end{figure}
We now construct a PDE that asymptotically converges to the harmonic
diffeomorphism, which we then discretize to obtain a distributed
pseudo-localization algorithm. We use the heat flow equation as the
basis to define the pseudo-localization algorithm, which yields a
harmonic map as its asymptotically stable steady-state solution.  We
begin by setting up the system for a stationary swarm, for which the
spatial domain is fixed.

Let $M \subset \real^2$ be a compact $2$-manifold with boundary, $N$
be the unit disk of $\real^2$, and $\mathbf{R} = (X,Y): M \rightarrow
N$. The heat flow
equation is given by:
\begin{align}
	\begin{aligned}
	\begin{cases} 
		 \partial_t X = \Delta X ,\\
		 \partial_t Y = \Delta Y ,
    \end{cases}
    &\text{for } \mathbf{r} \in \mathring{M}, \\
    \mathbf{R} = \xi, \hspace{0.15in} &\text{on } \partial M.
	\end{aligned}
        \label{eq:heat_flow}
\end{align}
The heat flow equation has been studied extensively in the
literature. For well-known existence and uniqueness results, we refer
the reader to~\cite{JE-LL:81}.

\begin{lemma}
\label{lemma:stage2_lemma}\longthmtitle{Pointwise convergence of the heat
  flow equation to a
  harmonic diffeomorphism}
The solutions of the heat flow equation~\eqref{eq:heat_flow} converge
pointwise to the harmonic map satisfying~\eqref{eq:harmonic_diffeo}, 
exponentially as $t \rightarrow \infty$, from any smooth
initial~$\mathbf{R}_0 \in H^1(M) \times H^1(M)$.
\end{lemma}
\begin{proof}
  Let~$\Psi$ be the solution
  to~\eqref{eq:harmonic_diffeo}, which is a harmonic map by
  definition. Let $\tilde{\mathbf{R}} = \mathbf{R} - \Psi$
  be the error where~$\mathbf{R} = (X,Y)$ is the solution to
  \eqref{eq:heat_flow}.  Subtracting \eqref{eq:harmonic_diffeo} from
  \eqref{eq:heat_flow}, we obtain:
\begin{align}
	\begin{aligned}
	\begin{cases} 
		 \partial_t X = \Delta X ,\\
		 \partial_t Y = \Delta Y ,
    \end{cases}
    &\text{for } \mathbf{r} \in \mathring{M}, \\
    \tilde{\mathbf{R}} = 0, \hspace{0.15in} &\text{on } \partial M.
	\end{aligned}
        \label{eq:heat_flow_error}
\end{align}
The Laplace operator~$\Delta$ with the Dirichlet boundary condition
in~\eqref{eq:heat_flow_error} is self-adjoint and has an infinite
sequence of eigenvalues $0 < \lambda_1 < \lambda_2 < \ldots$, with the
corresponding eigenfunctions $\lbrace \phi_i \rbrace_{i = 1}^{\infty}$ forming an
orthonormal  basis of $L^2(M)$ (where $\phi_i \in L^2(M)$ and $\Delta \phi_i = \lambda_i \phi_i$ for all~$i$, with $\phi_i=0$ on the boundary)
\cite{LCE:98}. Let the initial condition be $\tilde{X}_0 =
\sum_{i=1}^{\infty} a_i \phi_i$ and $\tilde{Y}_0 = \sum_{i=1}^{\infty}
b_i \phi_i$ (where~$a_i$ and~$b_i$ are constants for all~$i$). The solution to~\eqref{eq:heat_flow_error} is then given by $\tilde{X}(t,
\mathbf{r}) = \sum_{i=1}^{\infty} a_i e^{-\lambda_i t}
\phi_i(\mathbf{r})$ and $\tilde{Y}(t, \mathbf{r}) =
\sum_{i=1}^{\infty} b_i e^{-\lambda_i t}
\phi_i(\mathbf{r})$. Since
$\lambda_i > 0$, for all~$i$, we obtain $\lim_{t \rightarrow \infty}
\tilde{X}(t, \mathbf{r}) = 0$ and $\lim_{t \rightarrow \infty}
\tilde{Y}(t, \mathbf{r}) = 0$, for all $\mathbf{r} \in \bar{M}$.
Therefore, $ \lim_{t \rightarrow \infty} \mathbf{R}(t,\mathbf{r}) =
\Psi(\mathbf{r})$, for all $\mathbf{r} \in \bar{M}$, and the convergence is exponential.
\end{proof}

We now have a PDE that converges to the diffeomorphism given by
\eqref{eq:harmonic_diffeo} for the stationary case (agents in the
swarm are at rest). For the dynamic case, and to describe the
algorithm while the agents are in motion, we modify
\eqref{eq:heat_flow} as follows.  Let $\mathbf{R} = (X,Y): \real
\times \real^2 \rightarrow \real$. We are only interested in the
restriction to~$M(t)$, $\mathbf{R} |_{M(t)}$, at any time~$t$, so we
drop the restriction and just identify $\mathbf{R}\equiv
\mathbf{R}_{|_{M(t)}}$. Using the relation $\frac{dX}{dt} = \partial_t
X + \nabla X \cdot \mathbf{v}$, where $\mathbf{v}$ is a velocity
field, we obtain:
\begin{align}
	\begin{aligned}
	\begin{cases} 
		 \partial_t X = \Delta X - \nabla X \cdot \mathbf{v},\\
		 \partial_t Y = \Delta Y - \nabla Y \cdot \mathbf{v},
    \end{cases}
    &\text{for } \mathbf{r} \in \mathring{M}(t), \\
	\mathbf{R} = \xi, 
	\hspace{0.15in} &\text{on } \partial M(t).
	\end{aligned}
    \label{eq:heat_flow_dyn_swarm}
\end{align}
We now discretize \eqref{eq:heat_flow_dyn_swarm} to derive the
distributed pseudo-localization algorithm.  Now, we have $\rho: \real
\times \real^2 \rightarrow \real_{\geq 0}$ with support $M(t)$, the
density distribution of the swarm on the domain $M(t)$.
We view the swarm as a discrete approximation of the domain $M(t)$ with
density $\rho$, and the PDE
\eqref{eq:heat_flow_dyn_swarm} as approximated by a distributed
algorithm implemented by the swarm.

Here, we propose a candidate distributed algorithm, which would yield
the heat flow equation via a functional approximation.  Our candidate
algorithm is a time-varying weighted Laplacian-based distributed
algorithm, owing to the connection between the graph Laplacian and the
manifold Laplacian \cite{MB-PN:08}:
\begin{align}
  X_i(t+1) = X_i(t) + \sum_{j \in {\mathcal N}_i(t)}
  w_{ij}(t) (X_j (t) - X_i (t)),
	\label{eq:discrete_sum}
\end{align}
and a similar equation for~$Y$. We show how to derive next the values
for the weights $w_{ij}(t) \in \real$, for all $t$. First, the set of
neighbors, $j \in {\mathcal N}_i (t)$, of~$i$ at time~$t$, are the
spatial neighbors of~$i$ in $M(t)$, that is, ${\mathcal N}_i (t) = \{
j \in \Sc\,|\, \|\mathbf{r}_j (t) - \mathbf{r}_i(t)\| \le \epsilon \}
\equiv B_\epsilon(\mathbf{r}_i(t))$. Using $X_i(t+1) - X_i(t) =
\frac{d X}{d t} \delta t$, for a small $\delta t$, we make use of a
functional approximation of~\eqref{eq:discrete_sum}:
\begin{align}
  \frac{d X}{d t} \delta t =
  \int_{B_{\epsilon}(\mathbf{r}_i(t))} w(t, \mathbf{r}_i,\mathbf{s}) (
  X(t,\mathbf{s}) - X(t,\mathbf{r}_i) )~ \rho(t, \mathbf{s}) d\mu,
	\label{eq:continuum_sum}
\end{align}
where $d\nu = \rho~d\mu$ is a density-dependent measure on the
manifold, and the weighting function $w$ satisfies $w(t,
\mathbf{r}_i(t), \mathbf{r}_j(t)) = w_{ij}(t)$, for all $i,j\in
\Sc$. We note that the summation term in \eqref{eq:discrete_sum} is a
special form of the integral in \eqref{eq:continuum_sum} with a Dirac
measure $d\nu$ supported on the set $\lbrace \mathbf{r}_1(t), \ldots,
\mathbf{r}_N(t) \rbrace$ at time~$t$. Now, with the choice $w(t,
\mathbf{r}_i, \mathbf{s}) =
\frac{1}{\int_{B_{\epsilon}(\mathbf{s}(t))} \rho(t, \mathbf{\bar{s}})
  d\mu}$ and for very small $\epsilon$ (making
$\mathcal{O}(\epsilon^3)$ terms negligible), \eqref{eq:continuum_sum}
reduces to:
\begin{align*}
  \frac{d X}{d t} \delta t= a \Delta X,
\end{align*} 
where $a = \frac{1}{4 \epsilon} \int_{B_{\epsilon}(\mathbf{r}_i(t))}
(\mathbf{s} - \mathbf{r}_i(t)) \cdot (\mathbf{s} - \mathbf{r}_i(t))
d\mu$ is a constant. Now, with the choice $\delta t =
a$, we obtain:
\begin{align*}
  \frac{d X}{d t} = \frac{\partial X}{\partial t} + \mathbf{v} \cdot
  \nabla X = \Delta X,
\end{align*} 
which is the
PDE~\eqref{eq:heat_flow_dyn_swarm}. Let $d(t, \mathbf{r}_i(t)) =
\int_{B_{\epsilon}(\mathbf{r}_i(t))}\rho(t, \mathbf{s}) d \mu$ and
$d_i(t) = |\mathcal{N}_i (t)|$, for $i \in \Sc$. Substituting
$w_{ij}(t) = w(t, \mathbf{r}_i(t), \mathbf{r}_j(t)) =
\frac{1}{\int_{B_{\epsilon}(\mathbf{r}_j(t))} \rho(t,\mathbf{\bar{s}})
  d\mu} = \frac{1}{d(t, \mathbf{r}_j(t))} \approx \frac{1}{d_j(t)}$,
in~\eqref{eq:discrete_sum}, we get the distributed pseudo-localization
algorithm for the agents in the interior of the swarm to be:
\begin{align}
\begin{aligned}
  & X_i(t+1)
  = X_i(t) + \sum_{j \in {\mathcal N}_i(t)} \frac{1}{d_j(t)} (X_j(t) - X_i(t)), \\
  & Y_i(t+1) = Y_i(t) + \sum_{j \in {\mathcal N}_i(t)}
  \frac{1}{d_j(t)} (Y_j(t) - Y_i(t)).
\end{aligned}
\label{eq:discrete_pseudoloc_2D}
\end{align}
For the agents on the boundary~$\partial M(t)$, we have:
\begin{align*}
	\mathbf{R}_i = (X_i,Y_i) = \xi_i,
\end{align*}
where~$\xi_i = \xi(\mathbf{r}_i(t))$, for~$\mathbf{r}_i(t)
\in \partial M(t)$.  Note that the discretization scheme is
consistent, in that as the number of agents~$N \rightarrow \infty$,
the discrete equation~\eqref{eq:discrete_pseudoloc_2D} converges to
the PDE~\eqref{eq:heat_flow_dyn_swarm}.  In this way,
from~\eqref{eq:discrete_pseudoloc_2D}, the pseudo-localization
algorithm is a Laplacian-based distributed algorithm, with a
time-varying weighted graph Laplacian.

\subsection{Distributed density control law and analysis}
\label{sec:dist_control_2D}
In this section, we derive the distributed feedback control law to
converge to the desired density distribution over the target domain in
the two-dimensional case. The swarm dynamics are given by:
\begin{align}
	\begin{aligned}
          \partial_t \rho = - \nabla \cdot ( \rho \mathbf{v}),
          \hspace{0.15in} &\text{for } \mathbf{r} \in \mathring{M}(t), \\
          \partial_t \mathbf{r} = \mathbf{v}, \hspace{0.15in}
          &\text{on } \partial M(t).
	\end{aligned}
	\label{eq:2D_swarm}
\end{align}
\begin{assumption}\longthmtitle{Well-posedness of the PDE
    system}\label{as:wellposed2D}
  We assume that~\eqref{eq:2D_swarm} is well-posed, and that its
  solution~$\rho(t,\cdot)$ is sufficiently smooth and is bounded in the
  Sobolev space~$ H^1(\cup_t M(t))$, 
  the components of the velocity field~$\mathbf{v}$ are
  bounded in the Sobolev space~$ H^1(\cup_t M(t))$ and of the parametrized velocity on the boundary are bounded in 
  the Sobolev space~$H^1((0,1))$.
\end{assumption}
In what follows, we describe the control strategy based on three
different stages.

\subsubsection{Stage $1$}
In this stage, the objective is for the swarm to converge to the
target spatial domain $M^*$.

Let $\mathbf{r}^* : [0,1] \rightarrow \partial M^*$ be the closed
curve describing the desired boundary. Let $\mathbf{e}(\gamma) =
\mathbf{r}(\gamma) - \mathbf{r}^*(\gamma)$ be the position error of
agent~$\gamma$ on the boundary, where~$\mathbf{r}(\gamma)$ is the
actual position of agent~$\gamma$ computed as presented in
Section~\ref{sec:boundary_localization}. 
%
%
We define a distributed control law for swarm
motion as follows:
\begin{align}
	\begin{aligned}
		\begin{cases}
                  \mathbf{v} = - \frac{\nabla \rho}{\rho} , &\text{for } \mathbf{r} \in \mathring{M}(t), \\
           	  \partial_t \mathbf{v} = - \mathbf{e} - \mathbf{v},
                  \hspace{0.15in} &\text{on } \partial M(t).
		\end{cases}
	\end{aligned}
	\label{eq:stage1_vel}
\end{align}
\begin{theorem}\longthmtitle{Convergence to the desired spatial domain}
\label{thm:stage1_thm}
Under the well-posedness Assumption~\ref{as:wellposed2D}, the
domain~$M(t)$ of the system~\eqref{eq:2D_swarm}, with the distributed
control law~\eqref{eq:stage1_vel} converges to the target spatial
domain~$M^*$ as $t \rightarrow \infty$, from any initial domain~$M_0$
with smooth boundary.
\end{theorem}
\begin{proof}
  We consider an energy functional $E$ given by:
\begin{align*}
  E = \frac{1}{2} \int_{\partial M(t)} |\mathbf{e}|^2 + \frac{1}{2}
  \int_{\partial M(t)} |\mathbf{v}|^2 .
\end{align*}
Its time derivative, $\dot{E}$, using~\eqref{eq:stage1_vel}, is
given by:
\begin{align*}
  \dot{E} &= \int_{\partial M(t)} \mathbf{e} \cdot \mathbf{v} +
  \int_{\partial M(t)} \mathbf{v} \cdot \partial_t \mathbf{v} =
  \int_{\partial M(t)} (\mathbf{e} + \mathbf{v}) \cdot \partial_t
  \mathbf{v} = - \int_{\partial M(t)} |\mathbf{v}|^2.
\end{align*}
Clearly, $\dot{E} \leq 0$, and considering a parametrization of~$\partial M(t)$
by the interval~$[0,1)$, we have
~$\mathbf{v}(t,\cdot) \in
H^1((0,1))$ and bounded.  By Lemma~\ref{lemma:RKCT}, the
Rellich-Kondrachov Compactness theorem, $H^1((0,1))$ is
compactly contained in $L^2((0,1))$ (and we also have that $H^1((0,1))$ is
dense in $L^2((0,1))$). Thus, by the LaSalle
Invariance Principle, Lemma~\ref{lemma:LaSalle_inv}, we have that the
solutions to~\eqref{eq:2D_swarm} with the control
law~\eqref{eq:stage1_vel} converge in the~$L^2$-norm to the largest invariant subset
of~$\dot{E}^{-1} ( 0 )$, which satisfies:
\begin{align*}
  \lim_{t \rightarrow \infty} \| |\mathbf{v}| \|_{L^2(\partial M(t))} = 0, \hspace{0.2in}
  \lim_{t \rightarrow \infty} \partial_t \| |\mathbf{v}|
  \|_{L^2(\partial M(t))} = \lim_{t \rightarrow \infty} \int_{\partial
    M(t)} \mathbf{v} \cdot \partial_t \mathbf{v} = 0.
\end{align*}
The set $\dot{E}^{-1} ( 0 )$ is characterized by the first equality above
and the second equality is further satisfied by the invariant subset of $\dot{E}^{-1} ( 0 )$.
%
%
We know from \eqref{eq:stage1_vel} that $\partial_t \mathbf{v} = -
\mathbf{e} - \mathbf{v}$ on~$\partial M(t)$, which upon multiplying on
both sides by $\mathbf{v}$, integrating over~$\partial M(t)$ and
applying the previous equality on the integral of
$\mathbf{v}\cdot \partial_t \mathbf{v}$, yields $\lim_{t \rightarrow
  \infty} \int_{\partial M(t)} \mathbf{e} \cdot \mathbf{v} = 0$.  Now,
we have $|\partial_t \mathbf{v}|^2 = | \mathbf{e} |^2 + |\mathbf{v}
|^2 + 2 \mathbf{e} \cdot \mathbf{v}$, which on integrating over
$\partial M(t)$ yields $\lim_{t \rightarrow \infty} \| |\partial_t
\mathbf{\mathbf{v}}| \|_{L^2(\partial M(t))} = \lim_{t \rightarrow
  \infty} \| |\mathbf{e}| \|_{L^2(\partial M(t))}$.  By multiplying
$\partial_t \mathbf{v} = - \mathbf{e} - \mathbf{v}$ on both sides
by~$\partial_t \mathbf{v}$, integrating over $\partial M(t)$, and
using the Cauchy-Schwarz inequality, we obtain:
\begin{align*}
  \lim_{t \rightarrow \infty} \| |\partial_t\mathbf{ \mathbf{v}}|
  \|_{L^2(\partial M(t))}^2 &= \lim_{t \rightarrow \infty} -
  \int_{\partial M(t)} \mathbf{e} \cdot \partial_t \mathbf{v} \leq
  \lim_{t \rightarrow \infty} \int_{\partial
    M(t)} |\mathbf{e}| |\partial_t \mathbf{v}| \\
  &\leq \lim_{t \rightarrow \infty} \| |\mathbf{e}| \|_{L^2(\partial
    M(t))} \| |\partial_t\mathbf{ \mathbf{v}}| \|_{L^2(\partial M(t))}
  = \lim_{t \rightarrow \infty} \| |\partial_t \mathbf{\mathbf{v}}|
  \|_{L^2(\partial M(t))}^2
\end{align*}
In this way, the Cauchy-Schwarz inequality becomes an equality, which
implies that $\lim_{t \rightarrow \infty} \int_{\partial M(t)} \left[
  |\mathbf{e}| |\partial_t \mathbf{v}| - (-\mathbf{e})
  \cdot \partial_t \mathbf{v} \right] = 0$ (since the integrand is
non-negative and its integral is zero, it is zero almost everywhere),
thus $\lim_{t \rightarrow \infty} \partial_t \mathbf{v} = - \lim_{t
  \rightarrow \infty} \mathbf{e}$ almost everywhere (a.e.) on the
boundary, and, in turn, implies that $\lim_{t \rightarrow \infty}
\mathbf{v} = 0$ a.e. on the boundary (since $\partial_t \mathbf{v} = -
\mathbf{e} - \mathbf{v}$ and $\lim_{t \rightarrow \infty} \partial_t
\mathbf{v} = - \lim_{t \rightarrow \infty} \mathbf{e}$). From here,
and owing to the Invariance Principle, we have $\lim_{t \rightarrow
  \infty} \partial_t \mathbf{v} = 0 = \lim_{t \rightarrow \infty}
\mathbf{e}$ a.e. on the boundary.  Thus, we have that $\lim_{t
  \rightarrow \infty} M(t) = M^*$.
\end{proof}

\subsubsection{Stage $2$}
\label{sec:stage_2}
Here, the agents in the swarm implement the pseudo-localization
algorithm presented in Section~\ref{sec:2D_pseudoloc}. Since the
agents are distributed across the target spatial domain~$M^*$,
implementing the pseudo-localization algorithm yields the coordinate
transformation~$\Psi^*$ characteristic of the domain~$M^*$. We
therefore have~$\partial_t \Psi^* = 0$, which implies that~$\frac{d
  \Psi^*}{dt} = \partial_t \Psi^* + \nabla (\Psi^*) \mathbf{v} =
\nabla (\Psi^*) \mathbf{v}$, which will be used in Stage~$3$.

\subsubsection{Stage $3$}
In this stage, the boundary agents of the swarm remain stationary and
interior agents converge to the desired density distribution.

Consider the distributed control law, defined as follows for all
time~$t$:
\begin{align}
	\begin{aligned}
          \begin{cases}
            \frac{d \mathbf{v}}{dt} = - \rho \nabla(\rho - p^* \circ
            \Psi^*) + (\mathbf{v}\cdot \nabla) \mathbf{v} + \Delta \mathbf{v} -
            \mathbf{v}, \hspace{0.15in} &\text{for }
            \mathbf{r} \in \mathring{M}^*,\\
            \mathbf{v} = 0, \hspace{0.2in} &\text{on } \partial M^*,
	\end{cases}
	\end{aligned}
	\label{eq:dist_control_law_2D}
\end{align} 
where $\frac{d\mathbf{v}}{dt}$ at $\mathbf{r} \in M$ is the acceleration of the
agent at $\mathbf{r}$, the control input. Using the relation
$\frac{d}{dt} = \partial_t + \mathbf{v}\cdot \nabla$, it follows from
\eqref{eq:dist_control_law_2D} that $\partial_t \mathbf{v} = - \rho
\nabla(\rho - p^* \circ \Psi^*) + \Delta \mathbf{v} - \mathbf{v}$.
\begin{theorem}\longthmtitle{Convergence to the desired density}
\label{thm:stage3_thm}
The solutions~$\rho(t,\cdot)$ to~\eqref{eq:2D_swarm} 
 for the fixed
domain~$M^*$, under the distributed control law
\eqref{eq:dist_control_law_2D} and the well-posedness Assumption \ref{as:wellposed2D}, converge to~the desired density
distribution $\rho^*$ in the~$L^2$-norm as~$t \rightarrow \infty$.
\begin{proof}
We consider an energy functional $E$ given by:
\begin{align*}
  E = \frac{1}{2} \int_{M^*} |\rho - p^* \circ \Psi^*|^2 +
  \frac{1}{2} \int_{M^*} |\mathbf{v}|^2.
\end{align*}
Using Corollary~\ref{lemma:time_der}, to compute the derivative of
energy functionals, we obtain $\dot{E}$ (letting $\bar{\nabla}
= \left( \partial_X, \partial_Y \right)$) as follows:
\small
\begin{align*}
	\begin{aligned}	
          \dot{E}&= \int_{M^*} (\rho - p^* \circ \Psi^*)\left( \frac{d \rho}{dt} - \frac{d (p^* \circ \Psi^*)}{dt} \right) + \frac{1}{2} \int_{M^*} |\rho - p^* \circ \Psi^*|^2 \nabla \cdot \mathbf{v} + \int_{M^*} \mathbf{v} \cdot \partial_t \mathbf{v} \\
          &=  - \int_{M^*} (\rho - p^* \circ \Psi^*) \left( \rho \nabla \cdot \mathbf{v} + \bar{\nabla}p^* \cdot \frac{d\Psi^*}{dt} \right) + \frac{1}{2} \int_{M^*} |\rho - p^* \circ \Psi^*|^2 \nabla \cdot \mathbf{v} + \int_{M^*} \mathbf{v} \cdot \partial_t \mathbf{v} \\
          &= - \frac{1}{2}\int_{M^*} (\rho^2 - (p^* \circ \Psi^*)^2)
          \nabla \cdot \mathbf{v} - \int_{M^*} (\rho - p^* \circ
          \Psi^*) \bar{\nabla}p^* \cdot \frac{d\Psi^*}{dt} +
          \int_{M^*} \mathbf{v} \cdot \partial_t \mathbf{v},
	\end{aligned}
\end{align*} 
\normalsize
where, to
obtain the third equality, we expand the square $ |\rho - p^* \circ
\Psi^*|^2$ in the second integral of the second equality.  Since
$\mathbf{v} = 0$ on $\partial M^*$ and from Section~\ref{sec:stage_2},
we have $\frac{d\Psi^*}{dt} = \nabla (\Psi^*) \mathbf{v}$, we obtain:
\begin{align*}
  \begin{aligned}
    \dot{E}= \frac{1}{2}\int_{M^*} \nabla (\rho^2 - (p^* \circ
    \Psi^*)^2) \cdot \mathbf{v} - \int_{M^*} (\rho - p^* \circ
    \Psi^*) \bar{\nabla} p^* \cdot ( \nabla \Psi^* \mathbf{v}) +
    \int_{M^*} \mathbf{v} \cdot \partial_t \mathbf{v}.
	\end{aligned}
\end{align*}
We have $\bar{\nabla}p^* \nabla \Psi^* = \nabla(p^* \circ
\Psi^*)$, and $\nabla (\rho^2 - (p^* \circ \Psi^*)^2)  = (\rho - p^* \circ \Psi^*) \nabla (\rho + p^* \circ \Psi^*) 
+ (\rho + p^* \circ \Psi^*) \nabla (\rho - p^* \circ \Psi^*)$. Thus, we get:
\begin{align*}
  \begin{aligned}
    \dot{E}= \frac{1}{2}&\int_{M^*} (\rho + p^* \circ \Psi^*) \nabla
    (\rho - p^* \circ \Psi^*) \cdot \mathbf{v} + \frac{1}{2}\int_{M^*}
    (\rho - p^* \circ
    \Psi^*) \nabla (\rho + p^* \circ \Psi^*) \cdot \mathbf{v} \\
    &- \int_{M^*} (\rho - p^* \circ \Psi^*) \nabla (p^* \circ \Psi^*)
    \cdot \mathbf{v} + \int_{M^*} \mathbf{v} \cdot \partial_t
    \mathbf{v}.
	\end{aligned}
\end{align*}
We therefore get:
\begin{align*}
  \begin{aligned}
    \dot{E}&= \int_{M^*} \rho \nabla (\rho - p^* \circ
    \Psi^*) \cdot \mathbf{v} + \int_{M^*} \mathbf{v}
    \cdot \partial_t \mathbf{v} = \int_{M^*} \mathbf{v} \cdot \left(
      \rho \nabla (\rho - p^* \circ \Psi^*) + \partial_t
      \mathbf{v} \right).
	\end{aligned}
\end{align*}
From \eqref{eq:dist_control_law_2D}, we have $\partial_t \mathbf{v} =
- \rho \nabla (\rho - p^* \circ \Psi^*) + \Delta \mathbf{v} - \mathbf{v}$, and we obtain:
\begin{align*}
	\dot{E}= - \int_{M^*} |\mathbf{v}|^2 - \int_{M^*} |\nabla \mathbf{v}_x|^2 - \int_{M^*} |\nabla \mathbf{v}_y|^2 .
\end{align*}
Clearly, $\dot{E}\leq 0$, with~$\rho(t,.), \mathbf{v} \in H^1(M^*)$ and bounded (by Assumption~\ref{as:wellposed2D}).
By Lemma~\ref{lemma:RKCT}, the Rellich-Kondrachov Compactness
theorem, $H^1(M^*)$ is compactly contained in $L^2(M^*)$ (and we also know that the set of all
$(\rho, \mathbf{v})$ satisfying Assumption~\ref{as:wellposed2D} is dense in $L^2(M^*)$). Thus, by the
Invariance Principle, Lemma~\ref{lemma:LaSalle_inv}, we have that the
solution to~\eqref{eq:2D_swarm} converges in the~$L^2$-norm to the largest invariant
subset of~$\dot{E}^{-1} (0)$, which satisfies:
\begin{align}
	\begin{aligned}
          \| |\mathbf{v}| \|_{H^1(M^*)} = 0, \hspace{0.2in}
          \frac{1}{2} \partial_t \| |\mathbf{v}| \|_{L^2(M^*)}^2 =
          \int_{M^*} \mathbf{v} \cdot \partial_t \mathbf{v} = 0.
	\end{aligned}
	\label{eq:stage3_steadystate}
\end{align}
The set $\dot{E}^{-1} ( 0 )$ is characterized by the first equality
above and the second equality is further satisfied by the invariant
subset of $\dot{E}^{-1} ( 0 )$.  We know
from~\eqref{eq:dist_control_law_2D} that
\begin{align}
  \partial_t \mathbf{v} = - \rho \nabla (\rho - p^* \circ \Psi^*) + \Delta \mathbf{v} -
  \mathbf{v},
	\label{eq:stage3_vel_expression}
\end{align}
which substituted in~\eqref{eq:stage3_steadystate} yields~$\int_{M^*}
\rho \mathbf{v} \cdot \nabla(\rho - p^* \circ \Psi^*) = 0$.  Now,
from~\eqref{eq:stage3_vel_expression}, we obtain $\| |\partial_t
\mathbf{v}| \|_{L^2(M^*)}^2 = \int_{M^*} |\rho \nabla(\rho - p^* \circ
\Psi^*)|^2 + \int_{M^*} |\mathbf{v}|^2 + 2 \int_{M^*} \rho
\mathbf{v}\cdot \nabla(\rho - p^* \circ \Psi^*) = \int_{M^*} |\rho
\nabla(\rho - p^* \circ \Psi^*)|^2$; that is, $\| |\partial_t
\mathbf{v}| \|_{L^2(M^*)} = \| |\rho \nabla(\rho - p^* \circ \Psi^*)|
\|_{L^2(M^*)}$.  By multiplying~\eqref{eq:stage3_vel_expression}
by~$\partial_t \mathbf{v}$ on both sides and applying the
Cauchy-Schwarz inequality, we can also get that~$\| |\partial_t
\mathbf{v}| \|_{L^2(M^*)}^2 = - \int_{M^*} \rho \partial_t \mathbf{v}
\cdot \nabla(\rho - p^* \circ \Psi^*) \leq \int_{M^*} |\partial_t
\mathbf{v}| |\rho \nabla(\rho - p^* \circ \Psi^*)| \leq \| |\partial_t
\mathbf{v}| \|_{L^2(M^*)} | |\rho \nabla(\rho - p^* \circ \Psi^*)|
\|_{L^2(M^*)} = \| |\partial_t \mathbf{v}| \|_{L^2(M^*)}^2 $.  Thus,
the Cauchy-Schwarz inequality is in fact an equality, which implies
that~$\partial_t \mathbf{v} = - \rho \nabla(\rho - p^* \circ \Psi^*)$
almost everywhere in~$M^*$, which,
from~\eqref{eq:stage3_vel_expression} implies in turn that~$\mathbf{v}
= 0$ a.e.~in~$M^*$. It thus follows that $\partial_t \mathbf{v} = 0$
and~$\nabla(\rho - p^* \circ \Psi^*) = 0$ a.e in~$M^*$, and therefore
$\rho - p^* \circ \Psi^*$ is constant a.e.~in $M^*$.  Using the
Poincare-Wirtinger inequality, Lemma~\ref{lemma:poincare_wirtinger},
we obtain that $\|(\rho - p^* \circ \Psi^*) - (\rho - p^* \circ
\Psi^*)_{M^*} \| \leq C \| \nabla(\rho - p^* \circ \Psi^*)\| = 0$,
where $ (\rho - p^* \circ \Psi^*)_{M^*} = \frac{1}{|M^*|}\int_{M^*}
(\rho - p^* \circ \Psi^*)$. Since $\int_{M^*} \rho = \int_N p^* =
\int_{M^*} p^* \circ \Psi^* = 1$, we have that $ (\rho - p^* \circ
\Psi^*)_{M^*} = 0$, and therefore $\| \rho - p^* \circ \Psi^*
\|_{L^2(M^*)} = 0$. 
\end{proof}
\end{theorem}

\subsubsection{Robustness of the distributed control law}
The self-organization algorithm in~$2$D has been divided into three
stages, where asymptotic convergence is achieved in each stage (with
exponential convergence in the second stage).  We now present a
robustness result for convergence in Stage~$3$ under incomplete
convergence in the preceding stages.
\begin{lemma}\longthmtitle{Robustness of the control law}
  For every~$\delta > 0$, there exist~$T_1, T_2 < \infty$ such that
  when Stages~$1$ and~$2$ are terminated at~$t_1 > T_1$ and~$t_2 >
  T_2$ respectively, we have that $\lim_{t \rightarrow \infty} \|
  \rho(t,\cdot) - \rho^* \|_{L^2(M(t_1))} < \delta$.
\end{lemma}
\begin{proof}
  In Stage~$1$, it follows from Theorem~\ref{thm:stage1_thm} on the 
  convergence to the desired spatial domain
  that~$\lim_{t \rightarrow \infty} M(t) = M^*$. Then for
  every~$\epsilon_1 > 0$, we have~$T_1 < \infty$, such that~$d_H(M(t),
  M^*) < \epsilon_1$ for all~$t > T_1$, where~$d_H$ is the Hausdorff
  distance between two sets; see~\eqref{eq:defn_Hausdorff}. (Note that
  any appropriate notion of distance can alternatively be used here.)
  Let Stage~$1$ be terminated at~$t_1 > T_1$, which implies that the
  swarm is distributed across the domain~$M(t_1)$.
  In Stage~$2$, it follows from Lemma~\ref{lemma:stage2_lemma} on the
  convergence of the heat flow equation to the harmonic map, that for
  a domain~$M(t_1)$, we have that~$\lim_{t \rightarrow \infty}
  \mathbf{R}(t,\cdot) = \Psi_{M(t_1)}$ pointwise,
  where~$\Psi_{M(t_1)}$ is the harmonic map from~$M(t_1)$ to~$N$ (the
  unit disk).  Then, for every~$\epsilon_2 > 0$, we have a~$T_2 <
  \infty$, such that~$\| \mathbf{R}(t,\cdot) - \Psi_{M(t_1)}
  \|_{\infty} < \epsilon_2$ for all~$t > T_2$.  Let Stage~$2$ be
  terminated at~$t_2 > T_2$, which implies that the map from the
  spatial domain to the disk is~$\mathbf{R}(t_2,\cdot)$.
  In Stage~$3$, it follows from the arguments in the proof of
  Theorem~\ref{thm:stage3_thm} (on the convergence to the desired
  density distribution) that~$\lim_{t \rightarrow \infty}
  \rho(t,\cdot) = p^* \circ \mathbf{R}(t_2,\cdot)$ a.e.~in $M(t_1)$ if
  the map at the end of Stage~$2$ is $\mathbf{R}(t_2, \cdot)$.
  We characterize the error as~$\lim_{t \rightarrow \infty} \| \rho -
  \rho^* \|_{L^2(M(t_1))} = \| p^* \circ \mathbf{R}(t_2,\cdot) - p^*
  \circ \Psi^* \|_{L^2(M(t_1))} = \| p^* \circ \mathbf{R}(t_2,\cdot) -
  p^* \circ \Psi_{M(t_1)} +p^* \circ \Psi_{M(t_1)} - p^* \circ \Psi^*
  \|_{L^2(M(t_1))} \leq \| p^* \circ \mathbf{R}(t_2,\cdot) - p^* \circ
  \Psi_{M(t_1)} \|_{L^2(M(t_1))} + \| p^* \circ \Psi_{M(t_1)} - p^*
  \circ \Psi^* \|_{L^2(M(t_1))}$.  Recall that~$\|
  \mathbf{R}(t_2,\cdot) - \Psi_{M(t_1)} \|_{\infty} < \epsilon_2$, and
  since~$p^*$ is Lipschitz, we can get the bound $\| p^* \circ
  \mathbf{R}(t_2) - p^* \circ \Psi_{M(t_1)} \|_{L^2(M(t_1))} <
  \delta_1 = c \epsilon_2$ (where~$c$ is the Lipschitz constant times
  the area of~$M(t_1)$).  The harmonic map also depends continuously
  on its domain~\cite{AH:94}, which yields the bound $\|
  \Psi_{M(t_1)} - \Psi^* \|_{\infty} < \epsilon_3$, since~$d_H(
  M(t_1), M^* ) < \epsilon_1$. Thus, we get another bound $\| p^*
  \circ \Psi_{M(t_1)} - p^* \circ \Psi^*\|_{L^2(M(t_1))} < \delta_2 =
  c\epsilon_3$, and that~$\| \rho - \rho^* \|_{L^2(M(t_1))} < \delta_1
  + \delta_2 = \delta$.  Therefore, going backwards, for all~$\delta >
  0$, we can find~$T_1$ and~$T_2$ such that the density error is
  bounded by~$\delta$, when the Stages~$1$ and~$2$ are terminated at
  $t_1 > T_1$ and $t_2 > T_2$ respectively.
\end{proof}
\subsection{Discrete implementation}
\label{subsec:discrete_implementation_2D}

In this section, we present consistent schemes for discrete
implementation of the distributed control laws~\eqref{eq:stage1_vel}
and~\eqref{eq:stage3_vel_expression}, where the key aspect is the
computation of spatial gradients (of~$\rho$ in Stage~$1$, and
of~$\rho$, $\Psi^*$ and the components of velocity~$\mathbf{v}$ in
Stage~$3$).  The network graph underlying the swarm is a random
geometric graph, where the nodes are distributed according to the
density distribution over the spatial domain. According to this, every
agent communicates with other agents within a disk of given radius
(say~$r$) determined by the hardware capabilities, which reduces to
the graph having an edge between two nodes if and only if the nodes
are separated by a distance less than~$r$.  We recall the earlier
stated assumption that the agents know the true~$x$-
and~$y$-directions.

\subsubsection{On the computation of~$p^*$}
\label{sec:computingpstar_2D}
We first begin with an approach to compute offline the map~$p^*$ via
interpolation.  Let the desired domain $M^* \in \real^2$ be
discretized into a uniform grid to obtain $M^*_d = \lbrace
\mathbf{r}_1, \ldots, \mathbf{r}_m \rbrace$ (the centers of finite
elements, where $\mathbf{r}_k = (x_k,y_k)$).  The desired
density~$\rho^* : M^* \rightarrow \realpositive$ is known, and we
compute the value of~$\rho^*$ on~$M^*_d$ to get $\rho^*(\mathbf{r}_1,
\ldots, \mathbf{r}_m) = (\rho^*_1, \ldots, \rho^*_m)$.  We also have
$\Psi^*(x,y) = (X^*,Y^*) \in N$, for all $(x,y) \in M^*$.  Now,
computing the integral with respect to the Dirac measure for the
set~$M^*_d$, we obtain $\Psi^*(\mathbf{r}_1, \ldots, \mathbf{r}_m) =
(\Psi^*_1, \ldots, \Psi^*_m)$.  The value of the function~$p^*$ at
any~$(X,Y) \in N$ can be obtained from the relation $p^*(\Psi^*_1,
\ldots, \Psi^*_m) = \rho^*(\mathbf{r}_1, \ldots, \mathbf{r}_m)$ for $k
= 1, \ldots, m$ by an appropriate interpolation.

  \begin{figure}[H]
  \centering
  	\begin{tikzcd}
  		&(\rho^*_1, \ldots, \rho^*_m) = p^*(\Psi^*_1, \ldots, \Psi^*_m) \\
  		(\mathbf{r}_1, \ldots, \mathbf{r}_m) \arrow[ur, mapsto, shift left, end anchor = {[xshift = -1.5ex]}, "\rho^*"] \arrow[r, mapsto, end anchor = {[xshift = 6ex]}, "\Psi^*"] & \hspace{1cm} (\Psi^*_1, \ldots, \Psi^*_m) \arrow[u, Mapsto, shift right = 4ex, "p^*"]
  	\end{tikzcd}
  \caption*{Commutative diagram}
  \end{figure}

  \subsubsection{Discrete control law}
  As stated earlier, for the discrete implementation of the
  distributed control laws~\eqref{eq:stage1_vel}
  and~\eqref{eq:stage3_vel_expression}, the key aspect is the
  computation of spatial gradients (of~$\rho$ in Stage~$1$, and
  of~$\rho$, $\Psi^*$ and the components of velocity~$\mathbf{v}$ in
  Stage~$3$).  In the subsequent sections we present two alternative,
  consistent schemes for computing the spatial gradient (of any smooth
  function, with the above being the ones of interest), one using the
  Jacobian of the harmonic map and the other without it.

\subsubsection*{Computing the Jacobian of the harmonic map}
Let $J(\mathbf{r}) = \nabla \Psi (\mathbf{r})$ be the (non-singular)
Jacobian of the harmonic diffeomorphism $\Psi : M \rightarrow N$.
When the steady-state is reached in the pseudo-localization
algorithm~\eqref{eq:discrete_pseudoloc_2D} (i.e., $X_i(t+1) = X_i(t) =
\psi_1^i$ and $Y_i(t+1) = Y_i(t) = \psi_2^i$), we have,~$\forall\,i
\in \mathcal{S}$:
\begin{align*}
  \begin{aligned}
    \sum_{j \in {\mathcal N}_i} \frac{1}{d_j} (\psi_1^j - \psi_1^i) = 0,
    \qquad 
    \sum_{j \in {\mathcal N}_i} \frac{1}{d_j} (\psi_2^j - \psi_2^i) = 0,
	\end{aligned}
\end{align*}
where~$i$ is the index of the agent located at~$\mathbf{r} \in M$
and~$\mathcal{N}_i$ is the set of agents in a disk-shaped
neighborhood~$B_{\epsilon}(\mathbf{r})$ of area~$\epsilon$ centered
at~$\mathbf{r}$. Rewriting the above, we get, $\forall\,i \in
\mathcal{S}$:
\begin{align}
	\begin{aligned}
          \psi_1^i = \frac{\sum_{j \in {\mathcal N}_i} \frac{1}{d_j}
            \psi_1^j }{\sum_{j \in {\mathcal N}_i} \frac{1}{d_j} },
          \qquad \psi_2^i = \frac{\sum_{j \in {\mathcal N}_i}
            \frac{1}{d_j} \psi_2^j}{\sum_{j \in {\mathcal N}_i}
            \frac{1}{d_j}}.
	\end{aligned}
	\label{eq:SS_pseudoloc}
\end{align}
We assume that the agents have the capability in their hardware to
perturb the disk of communication~$B_{\epsilon}(\mathbf{r})$ (by
moving an antenna, for instance). The Jacobian~$J = \nabla \Psi$
, where~$\Psi = (\psi_1, \psi_2)$ 
is computed through
perturbations to~$\mathcal{N}_i$ (i.e., the
neighborhood~$B_{\epsilon}(\mathbf{r})$) and using consistent discrete
approximations:
\begin{align*}
	\begin{aligned}
          \partial_x \psi_1 \approx \frac{\psi_1 (\mathbf{r} + \delta x \mathbf{e}_1) -
            \psi_1 (\mathbf{r})}{\delta x}, \qquad
          \partial_y \psi_1 \approx \frac{\psi_1(\mathbf{r} + \delta y \mathbf{e}_2) -
            \psi_1 (\mathbf{r})}{\delta y},
	\end{aligned}
      \end{align*}  and
      similarly for~$\psi_2$. Now, $\psi_1(\mathbf{r}+\delta x \mathbf{e}_1)$ is computed as
      in~\eqref{eq:SS_pseudoloc} for~$\mathcal{N}^{\delta x}_i$, the
      set of agents in~$B_{\epsilon}(\mathbf{r}+\delta x \mathbf{e}_1)$ and
      $\psi_1(\mathbf{r}+\delta y \mathbf{e}_2)$
      from~$B_{\epsilon}(\mathbf{r}+\delta y \mathbf{e}_2)$.

      \subsubsection*{Computing the spatial gradient of a smooth
        function using the Jacobian of $\Psi$}
      Let $\nabla = \left( \partial_x, \partial_y \right)$ and
      $\bar{\nabla} = \left( \partial_{\psi_1}, \partial_{\psi_2}
      \right)$, where $\Psi = (\psi_1, \psi_2)$. We have $\partial_x =
      (\partial_x \psi_1) \partial_{\psi_1} + (\partial_x
      \psi_2) \partial_{\psi_2}$ and $\partial_y = (\partial_y
      \psi_1) \partial_{\psi_1} + (\partial_y
      \psi_2) \partial_{\psi_2}$. Therefore, $\nabla = J^{\top}
      \bar{\nabla}$. For a smooth function $f:M \rightarrow \real$, we
      have, $\nabla f= J^{\top} \bar{\nabla} f$, and the agents can
      numerically compute~$\bar{\nabla}$ by:
\begin{align*}
	\begin{aligned}
          \left( \frac{\partial f}{\partial \psi_1} \right)_i \approx
          \frac{1}{|\mathcal{N}_i|} \sum_{j \in \mathcal{N}_i}
          \frac{f_j - f_i}{\psi_1^j - \psi_1^i}, \qquad 
          \left( \frac{\partial f}{\partial \psi_2} \right)_i \approx
          \frac{1}{|\mathcal{N}_i|} \sum_{j \in \mathcal{N}_i}
          \frac{f_j - f_i}{\psi_2^j - \psi_2^i},
	\end{aligned}
\end{align*}
where~$i$ is the index of the agent located at~$\mathbf{r} \in M$
and~$\mathcal{N}_i$ is the set of agents in a ball
$B_\epsilon(\mathbf{r})$.

\subsubsection*{Computing the spatial gradient of a smooth function
  without the Jacobian of~$\Psi$}
In the absence of a Jacobian estimate, we use the following alternative
method for computing an approximate spatial gradient
estimate of a smooth function. This is
used in Stage~$1$ of the self-organization process.

Let $\bar{f}(\mathbf{r})$ be the mean value of~$f$ over a ball
$B_\epsilon(\mathbf{r})$:
\small
\begin{align*}
  \bar{f}(\mathbf{r}) = \frac{1}{\epsilon}
  \int_{B_{\epsilon}(\mathbf{r})} f d\mu \approx \frac{1}{|\mathcal{N}_i|}
  \sum_{j \in \mathcal{N}_i} f_j.
\end{align*}
\normalsize
We have: 
\small
\begin{align*}
  \frac{1}{\epsilon} \frac{\partial \bar{f}}{\partial x} &\approx
  \frac{1}{\epsilon} \frac{\bar{f}(\mathbf{r}+\delta x \mathbf{e}_1) -
    \bar{f}(x)}{\delta x} = \frac{1}{\epsilon}
  \frac{\int_{B_{\epsilon}(\mathbf{r}+\delta x \mathbf{e}_1)}
    f d\mu - \int_{B_{\epsilon}(\mathbf{r})} f d\mu}{\delta x} \\
  &= \frac{1}{\epsilon}  \int_{B_{\epsilon}(\mathbf{r})} \frac{(f(\mathbf{r}+\delta x \mathbf{e}_1) - f(\mathbf{r})) }{\delta x}d\mu 
  \approx \frac{1}{\epsilon}  \int_{B_{\epsilon}(\mathbf{r})}  \frac{\partial f}{\partial x} d\mu 
  = \overline{\left( \frac{\partial f}{\partial x} \right)}.
\end{align*}
\normalsize
Similarly,
\small
\begin{align*}
  \frac{1}{\epsilon} \frac{\partial \bar{f}}{\partial y} \approx
  \frac{1}{\epsilon} \frac{\bar{f}(\mathbf{r}+\delta y \mathbf{e}_2) -
    \bar{f}(x)}{\delta y} \approx \overline{\left( \frac{\partial
        f}{\partial y} \right)}.
\end{align*}
\normalsize
In all, for any scalar function $f$, each agent can use the approximation:
\small
\begin{align}\label{eq:gradientapprox}
 (\nabla f)_i &\approx \left( \overline{\left( \frac{\partial f}{\partial x}
    \right)}, \overline{\left( \frac{\partial f}{\partial y} \right)}
\right) = \frac{1}{\epsilon} \left( \frac{\partial \bar{f}}{\partial
    x}, \frac{\partial \bar{f}}{\partial y} \right),
\end{align}
\normalsize
to estimate of the gradient $\nabla f$. 

\subsubsection{On the convergence of the discrete system}
We have noted earlier that the pseudo-localization
algorithm~\eqref{eq:discrete_pseudoloc_2D} satisfies the consistency
condition in that as~$N \rightarrow \infty$,
Equation~\eqref{eq:discrete_pseudoloc_2D} converges to the
PDE~\eqref{eq:heat_flow_dyn_swarm}.  The pseudo-localization algorithm
is also essentially a weighted Laplacian-based distributed algorithm
that is stable. Thus, by the Lax Equivalence theorem~\cite{GDS:85},
the solution of~\eqref{eq:discrete_pseudoloc_2D} converges to the
solution of~\eqref{eq:heat_flow_dyn_swarm} as $N \rightarrow \infty$.
However, for the distributed control laws in Stages~$1$-$3$, we are
only able to provide consistent discretization schemes. The dynamics
of the swarm~\eqref{eq:2D_swarm} with the control
laws~\eqref{eq:stage1_vel} and~\eqref{eq:dist_control_law_2D} are
nonlinear for which is no equivalent convergence theorem. Further
analysis to determine convergence is required, which falls out the
scope of this present work.

\begin{algorithm}
\footnotesize
\label{alg:2D_implementation}
\begin{algorithmic}[1]
\caption{Self-organization algorithm for 2D environments}
	\State \textbf{Input:} $M^*$, $\rho^*$ and $k_1$, $k_2$, $K$
        (number of iterations for each stage), $\Delta t$ (time step)
        \State \textbf{Requires:}
	\State $\quad$Offline computation of~$p^*$ similar to the outline in
        Section~\ref{sec:computingpstar_1D}
        \State $\quad$Boundary agents are aware of being at boundary
        or interior of domain,
            can\\
            $\qquad$communicate with others along the boundary,
            can approximate the normal\\ 
            $\qquad$to the boundary, and can measure density of boundary agents, 
        \State $\quad$Agents have knowledge of a common orientation of a reference frame
    \State \textbf{Initialize:} $\mathbf{r}_i$ (Agent positions), $\mathbf{v}_i = 0$ (Agent velocities) 
	\State Boundary agents localize as outlined in Section~\ref{sec:boundary_localization}
	\State \textbf{Stage~$1$:}
		\For{$k:=1$ to $k_1$} 
			\If{agent $i$ is at the interior of domain}	
				\State compute
                                $\mathbf{v}_i(k) = -\frac{(\nabla
                                  \rho)_i}{\rho_i}(k)$ 
                                from~\eqref{eq:stage1_vel} 
                                \State move $\mathbf{r}_i(k+1) =
                                \mathbf{r}_i(k) + \mathbf{v}_i(k) \Delta t$
			\ElsIf {agent $i$ is at the boundary of domain}
				\State compute $\mathbf{v}_i (k+1) = \mathbf{v}_i (k)
                                -(\mathbf{r}_i(k) - \mathbf{r}_i^*(k) + \mathbf{v}_i (k)) \Delta t$
                                from~\eqref{eq:stage1_vel}, 
                                and move $\mathbf{r}_i(k+1) = \mathbf{r}_i(k) + \mathbf{v}_i(k) \Delta t$
			\EndIf
		\EndFor
	\State \textbf{End Stage~$1$}
	\State \textbf{Stage~$2$:}
	\State Boundary agents map themselves onto unit circle 
        according to~\eqref{eq:boundary_maponto_circle}
	\For{$k:=1$ to $k_2$} 	
			\For{agent $i$ in the interior}
                        \State compute $X_i(k+1)$, $Y_i(k+1)$
                        according to~\eqref{eq:discrete_pseudoloc_2D}
                        \EndFor
	\EndFor
	\State \textbf{Stage~$3$:} \For{$k:=1$ to $K$} \For{agent $i$
          in the interior} 
        \State compute $\mathbf{v}_i(k+1) = \mathbf{v}_i(k) + ( - \rho_i(k) (\nabla (\rho - p^* \circ \Psi^*))_i(k) +
        (\mathbf{v}_i(k) \cdot \nabla) \mathbf{v}_i(k) - \mathbf{v}_i(k) ) \Delta t $ from~\eqref{eq:dist_control_law_2D}
        , with $(\nabla (\rho - p^*  \circ \Psi^*))_i(k)$  as  in~\eqref{eq:gradientapprox}
        \State update $\mathbf{r}_i(k+1) = \mathbf{r}_i(k) + \mathbf{v}_i(k) \Delta t$
                       \EndFor
	\EndFor
\end{algorithmic}
\end{algorithm}
\section{Numerical simulations}
\label{sec:simulation}
In this section, we present numerical simulations of swarm
self-organization, that is, of the control laws presented in
Sections~\ref{subsec:distributed control} and of
Section~\ref{sec:dist_control_2D}.
\subsection{Self-organization in one dimension}
In the simulation of the~$1$D case, we consider a swarm of $N = 10000$
agents, the desired density distribution is given by $\rho^*(x) =
a\sin(x) + b$, where $a = 1 - \frac{\pi}{2N}$ and $b = \frac{1}{N}$,
$x \in \left[0, \frac{\pi}{2} \right]$. We use a kernel-based method
to approximate the continuous density function, which is given by:
\begin{align*}
  \rho(t, \mathbf{r}) = \sum_{i \in {\Sc}} K \left( \frac{\|
      \mathbf{r} - \mathbf{r}_i(t) \|}{d} \right), \hspace{0.2in}
	K(x)= 
        \begin{cases} 
          \frac{c_d}{d^n},& \text{for } 0 \leq x<1 , \\
          0,& \text{for } x \geq 1 ,
        \end{cases}
\end{align*} 
is a flat kernel and $c_d \in \realpositive$ is a constant
\cite{YC:95}. We discretize the spatial domain with~$\Delta x = 0.001$
units, and use an adaptive time step. The self-organization begins
from an arbitrary initial density
distribution. Figure~\ref{fig:simulation_1} shows the initial density
distribution, an intermediate distribution and the final
distribution. We observe that there is convergence to the desired
density distribution, even with noisy density measurements.
\begin{figure}[!h]
 	\begin{minipage}[c]{0.6\textwidth}
	\includegraphics[width=1\linewidth]{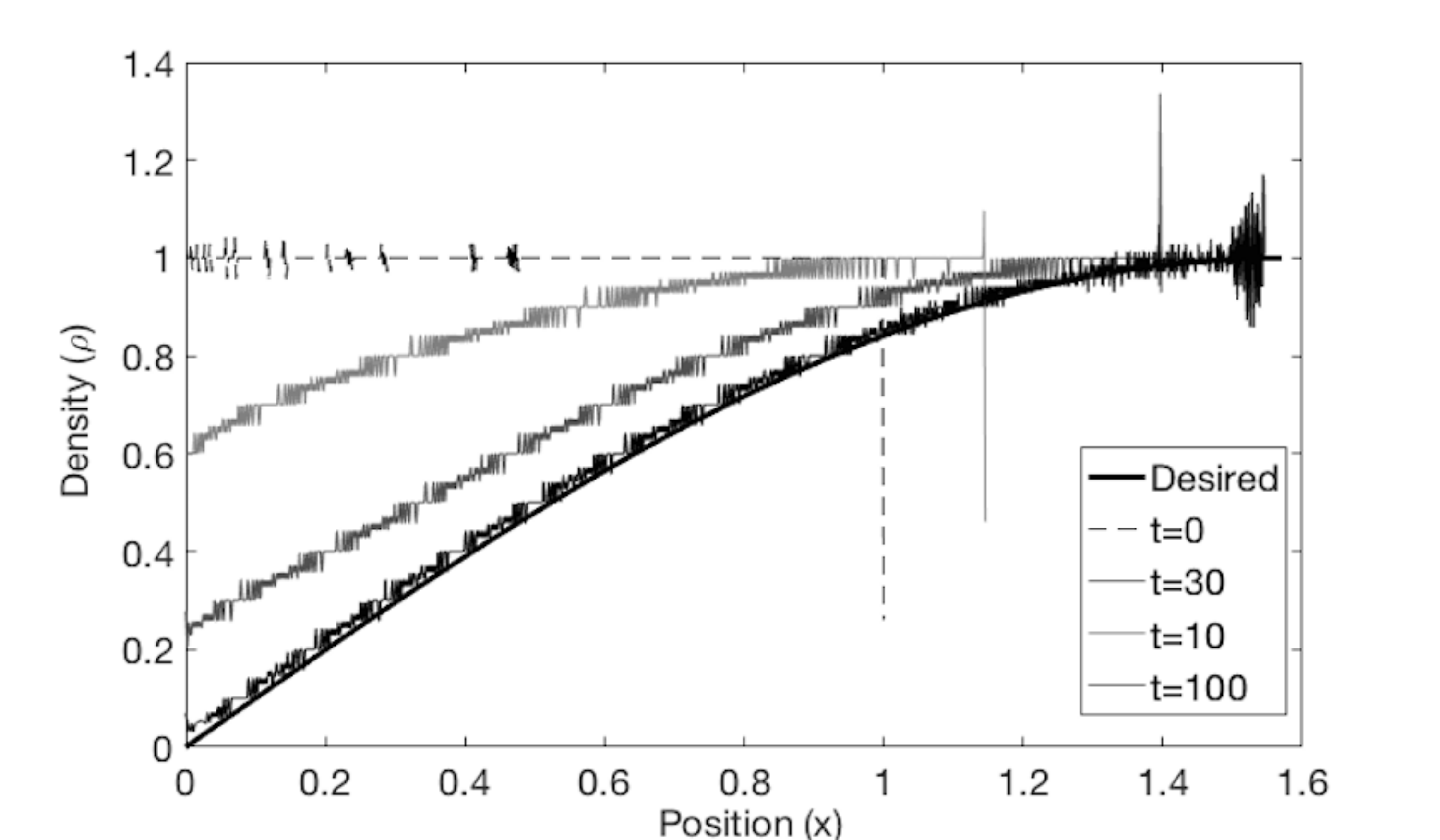}
	\end{minipage}
	\begin{minipage}[c]{0.3\textwidth}
	\captionsetup{justification=centering}
	\caption{Density $\rho(x)$ plotted against position $x$ at
          different instants of time.}
          \label{fig:simulation_1}
      \end{minipage}
\end{figure}
\subsection{Self-organization in two dimensions}
In the simulation of the~$2$D case, we first present in
Figure~\ref{fig:stage1fig} the evolution of the boundary of the swarm
in Stage~$1$, where the swarm converges to the target spatial
domain~$M^*$ from an initial spatial domain. The
target spatial domain, a circle of radius~$0.5$ units, given by $M^* =
\lbrace (x,y) \in \real^2 \,|\, (x-0.6)^2 + y^2 \leq 0.25 \rbrace$, with
the desired density distribution~$\rho^*$ given by $\rho^*(x,y) =
\frac{1}{\left( (x-0.4)^2 + y^2 \right)^{0.3}}$.
\begin{figure}[!h]
	\begin{center}
	\includegraphics[width=1\textwidth]{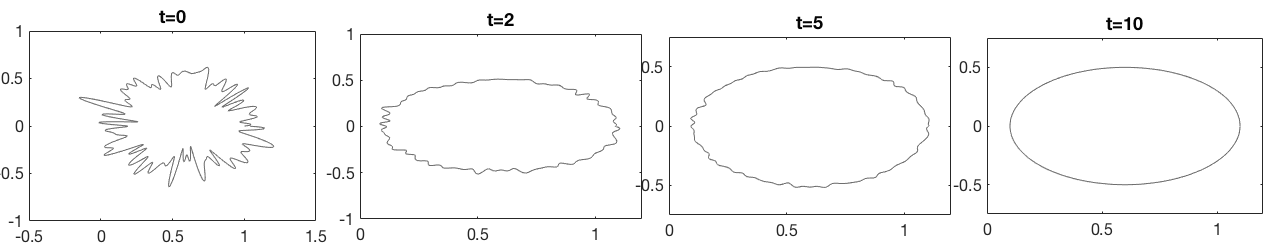}
	\captionsetup{justification=centering}
	\caption{Evolution of the swarm boundary in Stage~$1$.}
	\label{fig:stage1fig}
	\end{center}
\end{figure}

We present in Figures~\ref{fig:stage2X} and~\ref{fig:stage2Y} 
the result of implementation of the pseudo-localization algorithm
with the steady state distributions of~$\Psi^* = (\psi^*_1, \psi^*_2)$
respectively. We note that the steady state distribution~$\Psi^*$
as a function of the spatial coordinates $(x,y)$ in this case is linear.
\begin{figure}[!h]
\centering
\begin{minipage}{0.5\textwidth}
  \centering
  \includegraphics[width=1\linewidth]{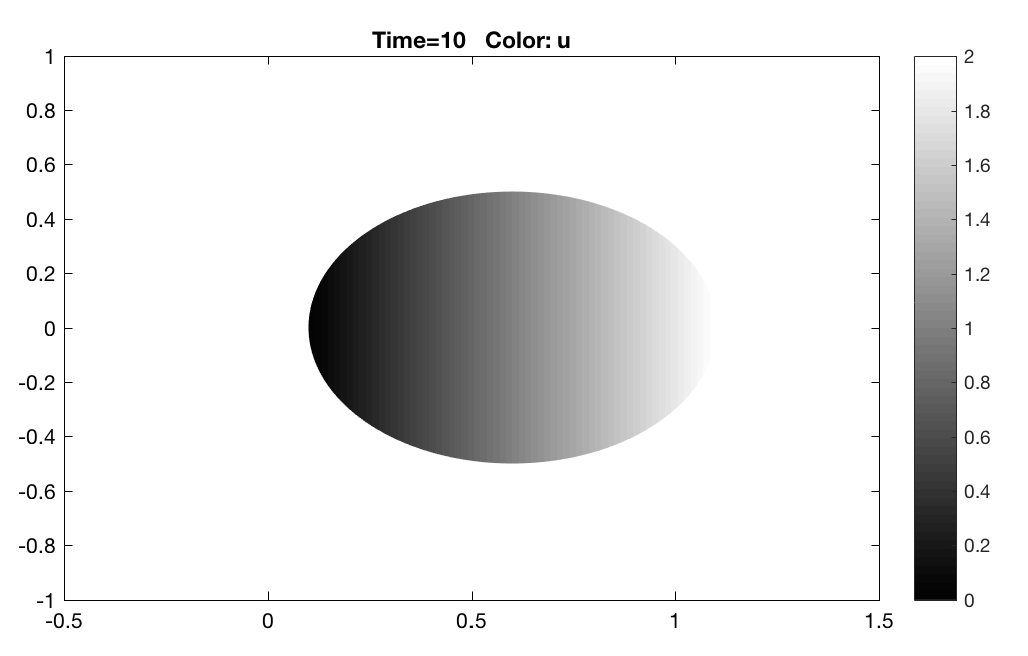}
  \captionof{figure}{Steady-state distribution of~$\psi^*_1$.}
  \label{fig:stage2X}
\end{minipage}%
\begin{minipage}{0.5\textwidth}
  \centering
  \includegraphics[width=1\linewidth]{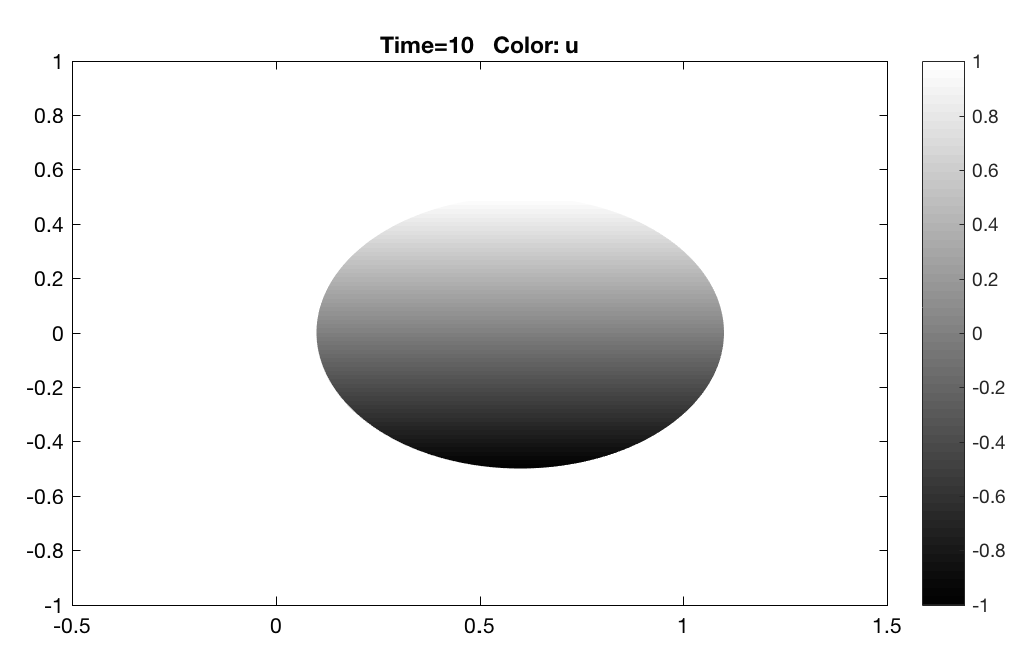}
  \captionof{figure}{Steady-state distribution of~$\psi^*_2$.}
  \label{fig:stage2Y}
\end{minipage}
\end{figure}

Next, we focus on Stage~$3$ of the self-organization process, where
the agents already distributed over the target spatial domain,
converge to the desired density distribution. The initial density distribution
of the swarm is uniform, and the distributed control law of Stage~$3$
in Section~\ref{sec:dist_control_2D}
is implemented.
Figure~\ref{fig:snapshots} shows the density distribution at a few
intermediate time instants of implementation and
figure~\ref{fig:density_error_plor} shows the spatial density error
plot, where~$e(\rho) = \int_{M^*} |\rho - \rho^*|^2$ is the spatial
density error. The results show convergence as desired. 
%
%
\begin{figure}[!h]
\begin{minipage}{0.5\textwidth}
\centering
	\includegraphics[width=1\linewidth]{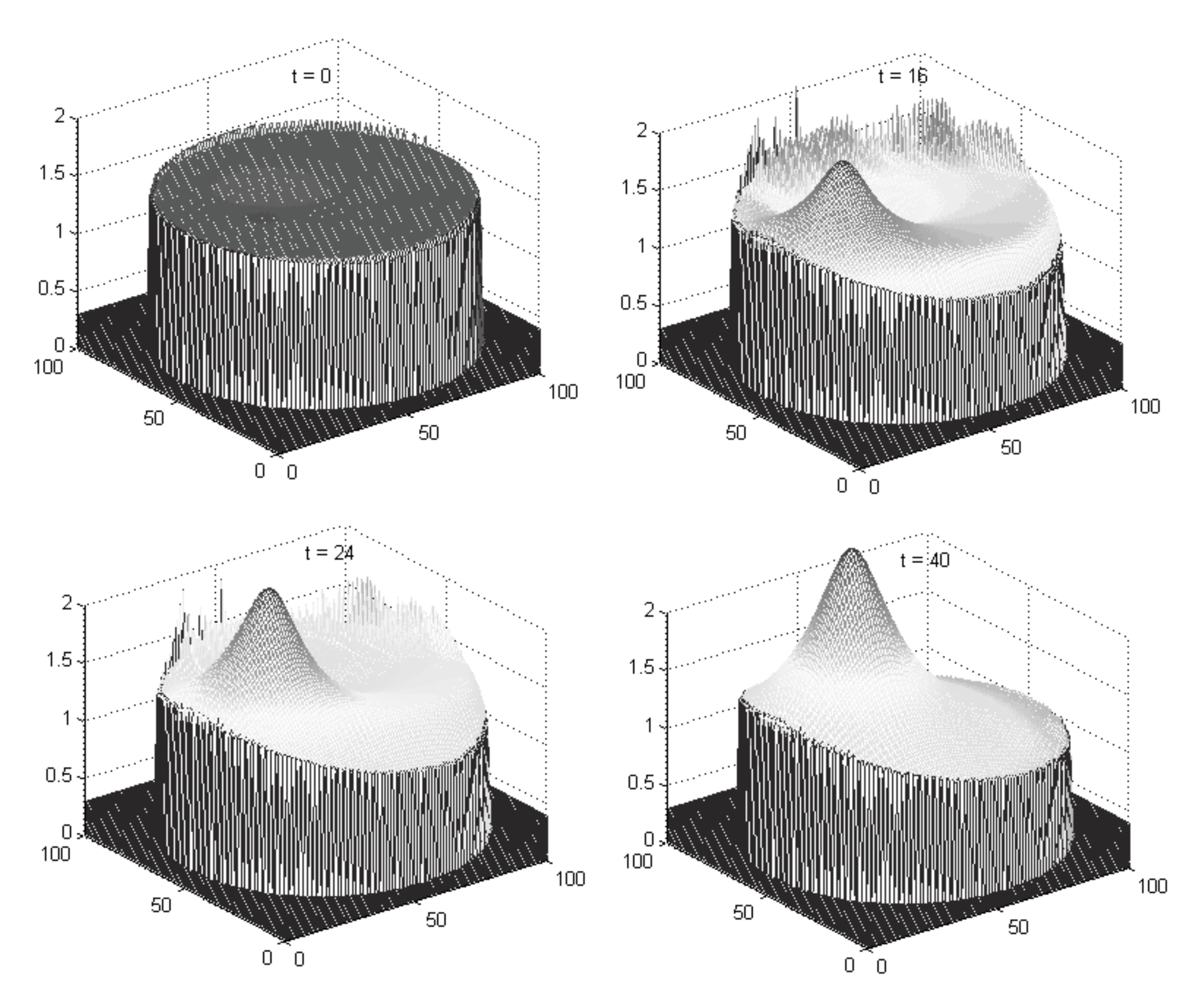}
	\captionsetup{justification=centering}
	\caption{Evolution of density distribution in Stage~$3$.}
	\label{fig:snapshots}
\end{minipage}
\begin{minipage}{0.45\textwidth}
\centering
	\includegraphics[width=1\linewidth]{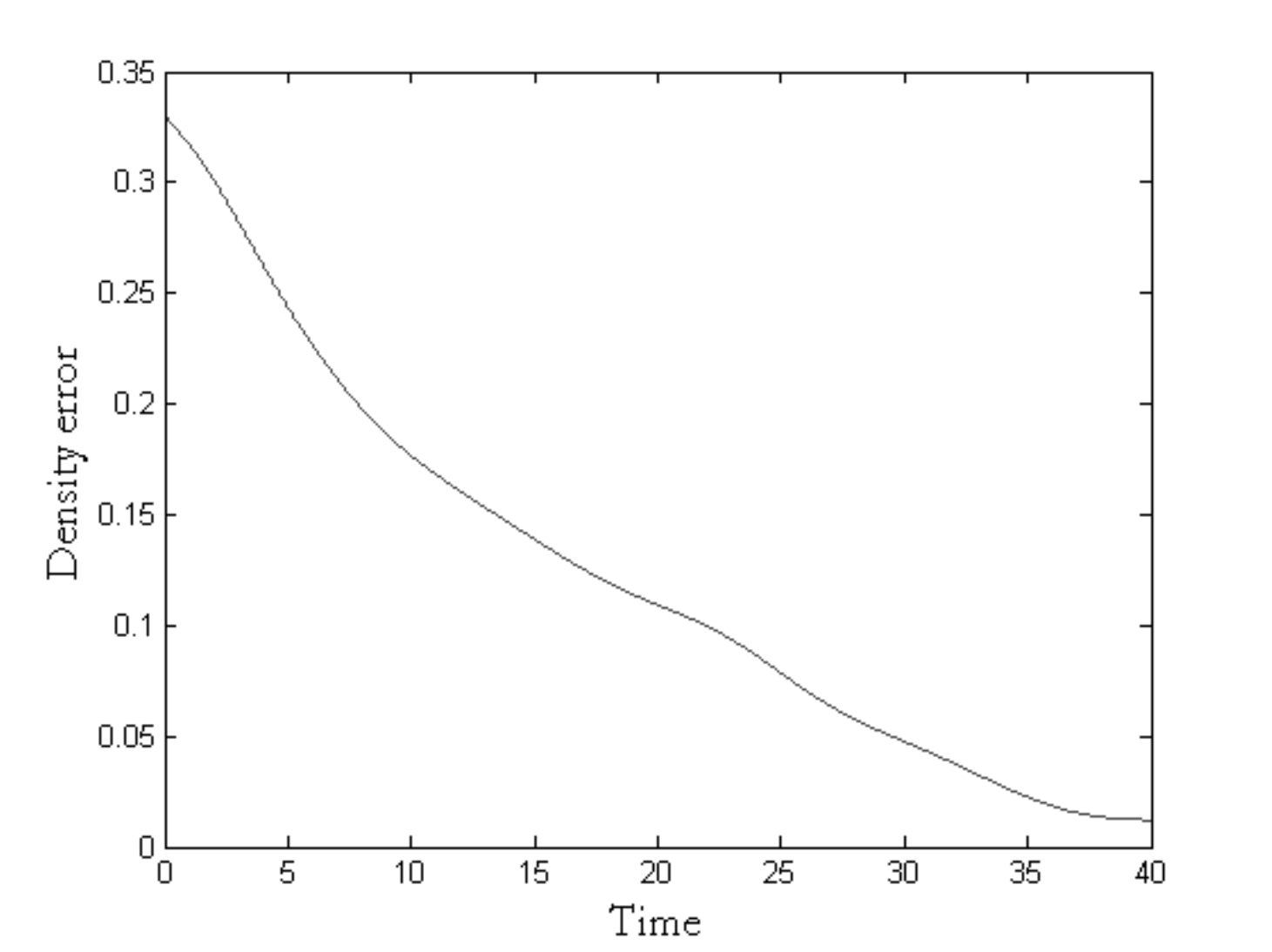}
	\captionsetup{justification=centering}
	\caption{Spatial density error $e(\rho) = \int_{M^*} |\rho -
          \rho^*|^2$ vs time,}
	\label{fig:density_error_plor}
\end{minipage}
\end{figure}

\section{Conclusions}
\label{sec:conclusions}
In this paper, we considered the problem of self-organization in
multi-agent swarms, in one and two dimensions, respectively. The
primary contribution of this paper is the analysis and design of
position and index-free distributed control laws for swarm
self-organization for a large class of configurations. This was
accomplished through the introduction of a distributed
pseudo-localization algorithm that the agents implement to find their
position identifiers, which then use in their control laws.  The
validation of the results for more general non-simply connected
domains will be considered in the future.  An extension to this work
will involve the characterization of constraints on the local density
function to capture finite robot sizes and collision avoidance
constraints, as well as accounting for possible non-holonomic
constraints on the motion of the robots.

\section*{Acknowledgments}
The authors would like to thank Prof. Lei Ni at the UC San Diego
Mathematics Department and the reviewers of this manuscript for their
valuable inputs.
\bibliographystyle{siamplain}
\bibliography{alias,SMD-add,SM,JC,FB}

\end{document}